\documentclass[12pt, reqno]{amsart}
\usepackage{amsmath, amsthm, amscd, amsfonts, amssymb, graphicx, color, mathrsfs}
\usepackage[bookmarksnumbered, colorlinks, plainpages]{hyperref}
\usepackage[all]{xy} 
\usepackage{slashed}

\newtheorem{theorem}{Theorem}[section]
\newtheorem{lemma}[theorem]{Lemma}

\theoremstyle{definition}
\newtheorem{definition}[theorem]{Definition}
\newtheorem{example}[theorem]{Example}

\theoremstyle{remark}
\newtheorem{remark}[theorem]{Remark}
\numberwithin{equation}{section}

\begin{document}
\setcounter{page}{1}

\title[Estimates for the full maximal function on graded Lie groups]{Estimates for the full maximal function on graded Lie groups}

\author[D. Cardona]{Duv\'an Cardona}
\address{
  Duv\'an Cardona:
  \endgraf
  Department of Mathematics: Analysis, Logic and Discrete Mathematics
  \endgraf
  Ghent University, Belgium
  \endgraf
  {\it E-mail address} {\rm duvanc306@gmail.com, duvan.cardonasanchez@ugent.be}
  }

\thanks{ The author was supported by the Research Foundation-Flanders
(FWO) under the postdoctoral
grant No 1204824N.  The author was supported  by the FWO  Odysseus  1  grant  G.0H94.18N:  Analysis  and  Partial Differential Equations and by the Methusalem programme of the Ghent University Special Research Fund (BOF)
(Grant number 01M01021). 
}


\begin{abstract} On $\mathbb{R}^n,$ a classical result due to Bourgain establishes the restricted weak $(\frac{n}{n-1},1)$ inequality for the full maximal function $M_F^{d\sigma}$  associated to the spherical averages. In this work we present an extension to Bourgain's result on graded Lie groups for a family of full maximal operators. We formulate this extension using the group Fourier transform of the measures under consideration and the symbols of (positive Rockland operators which are) positive left-invariant hypoelliptic partial differential operators on graded Lie groups. 
\end{abstract} 

\maketitle

\allowdisplaybreaks
\tableofcontents

\section{Introduction}
The aim of this work is to formulate certain boundedness properties for the full maximal function on graded Lie groups. Our main goal is to extend the restricted weak $(\frac{n}{n-1},1)$ inequality due to Bourgain  \cite{Bourgain1985} for the full maximal function $M_F^{d\sigma}$ associated to the spherical averages, to a family of full maximal functions on graded Lie groups.  As a consequence of interpolating our (restricted weak estimate in the form of a certain) $L^{p,1}\rightarrow L^{p,\infty}$ boundedness theorem with the corresponding $L^\infty$-boundedness result we immediately obtain the $L^{p,q}$-boundendess of these maximal operators on some Lorentz spaces.  Consequently, when $p=q$ we provide an extension to the $L^p$-boundedness result of $M_F^{d\sigma}$ due to Stein \cite{Stein1976} for the spherical averages when $p>\frac{n}{n-1},$ to the setting of graded Lie groups. 

Our main result is Theorem \ref{Maximal:Function:Graded} below. Our approach will be to identify the information on the problem coming from both, harmonic analysis and representation theory of nilpotent Lie groups. Since we consider graded Lie groups, several techniques from operator theory will be used to deal with the spectral calculus of Rockland operators (which are positive left-invariant hypoelliptic partial differential operators). 

\subsection{Short review of the Euclidean setting} To present our main result we briefly discuss the Euclidean setting. 
Consider the full maximal function on $\mathbb{R}^n$
\begin{equation}\label{Full}
   {M}^{d\sigma}_{F}f(x)=\sup_{r>0}\left|\smallint\limits_{\mathbb{S}^{n-1}}f(x-ry)d\sigma(y)\right|,\,n\geq 2,
\end{equation}where $d\sigma$ is the surface measure on the sphere $\mathbb{S}^{n-1}.$ For $n\geq 3,$  Stein in \cite{Stein1976} proved the boundedness of
\begin{equation}\label{Stein:1976}
{M}^{d\sigma}_{F}:L^p(\mathbb{R}^n)\rightarrow L^p(\mathbb{R}^n)
\end{equation}
for all $p>\frac{n}{n-1},$ and  Bourgain in \cite{Bourgain1985} also proved that
\begin{equation}\label{Bourgain:restricted}
    {M}^{d\sigma}_{F}:L^{\frac{n}{n-1},1}(\mathbb{R}^n)\rightarrow L^{\frac{n}{n-1},\infty}(\mathbb{R}^n)
\end{equation}
is bounded. Here we recall that the boundedness property in \eqref{Bourgain:restricted} is equivalent to the restricted weak $(\frac{n}{n-1},1)$ boundeness of $M_{F}^{d\sigma}.$ We also observe that Stein in \cite{Stein1976} proved that $M_F^{d\sigma}$ is unbounded
when $p\leq \frac{n}{n-1},$ and $n \geq  2$. Then, Bourgain \cite{Bourgain1986} proved  the boundedness of $M_{F}^{d\sigma},$ when $n=2.$ More precisely, Bourgain showed that $$ {M}^{d\sigma}_{F}:L^p(\mathbb{R}^2)\rightarrow L^p(\mathbb{R}^2) $$ is bounded for all $p>2.$ In view of the boundedness theorem in \eqref{Bourgain:restricted} it is also natural to ask if $M_{F}^{d\sigma}$ satisfies restricted weak inequalities when $n=2,$ as remarked by Wolff \cite{Wolff1996}. Then, solving  Wolff's question in  \cite{STW2003}, Seeger, Tao, and Wright have disproved the restricted weak type (2,2) inequality for $n=2.$ 

Several other
authors have also studied the spherical maximal function. It would be difficult to review the complete literature here, especially with the recent developments in the multilinear setting. Nevertheless, we provide some references, for instance, we refer to  Calder\'on \cite{Calderon1979}, Carbery \cite{Carbery}, Cowling and Mauceri \cite{CowlingMauceri},  Dosidis and Grafakos \cite{DosidisGrafakos}, Duoandikoetxea and Vega \cite{DuoandikoetxeaVega1996}, Greenleaf \cite{Greenleaf1981}, Mockenhaupt, Seeger, and Sogge \cite{MockenhauptSeegerSogge1992} and  Schlag \cite{Schlag1998}. Further references about the subject will be given in Subsection \ref{References:section}. 

Having presented a short review of the Euclidean setting, in the next section, we introduce the full maximal function in the setting of homogeneous nilpotent Lie groups to present our main theorem in Subsection \ref{Subsection:mainresult}. 
\subsection{The full maximal function $M_{F}^{d\sigma}$ and admissible measures}
We consider the full maximal function
\begin{equation}\label{defin:maximal:function}
    M_{F}^{d\sigma}f(x)=\sup_{t>0}\left|\smallint\limits_{G}f\left(x\cdot (t\cdot y)^{-1}\right)d\sigma(y)\right|,
\end{equation} associated to an arbitrary finite Borel measure $d\sigma$   with compact support on $G.$
In Definition \ref{Admissible:measure} below we describe the measures $d\sigma$ considered in this work. 

Our main motivation is to reformulate the {\it curvature conditions} available in the Euclidean setting in terms of group Fourier transform type conditions. One reason for this is that these two approaches are equivalent in the Euclidean setting, see e.g \cite[Page 622]{STW}.  Here, $\widehat{G}$ denotes the unitary dual of a graded Lie group $G,$ and  $\pi\mapsto\widehat{d\sigma}(\pi)$ denotes the operator-valued group Fourier transform of $d\sigma$ at $\pi\in \widehat{G}.$ We refer the reader to Section \ref{preliminaries} for details about the Fourier analysis on graded Lie groups.  A concise presentation about the properties of Rockland operators can be found in  \cite[Chapter 4]{FischerRuzhanskyBook}. In Section \ref{Motivation} we motivate the {\it curvature conditions} (ii) and (iii) of Definition \ref{Admissible:measure}  taking intuition from Fefferman's discussion in \cite[Page 527]{Fefferman2020:Stein:memory}.
 
\begin{definition}[Admissible measures]\label{Admissible:measure}  Let $d\sigma$ be a finite Borel measure of compact support on a graded Lie group $G$ of homogeneous dimension $Q.$ Let $\mathcal{R}$ be a positive Rockland operator on $G$ of homogeneous degree $\nu>0.$

We say that   $d\sigma$  is $(Q,Q_0,\varepsilon_0,a)$-{\it admissible} (or {\it admissible} for short) if it satisfies the topological property (i) and the ``curvature properties" (ii) and (iii) below. 
\begin{itemize}
    \item[(i)] Let $K$ be the support of $d\sigma.$ For any $\beta>0,$ and each $x\in G,$ the measure $\sigma(\mathbb{S}_\beta)$ of the set
    \begin{equation}\label{Q:o}
    \mathbb{S}_\beta(x)= \{y\in K:|xy^{-1}|\sim \beta \},
\end{equation} satisfies the growth estimate\footnote{Note that in the case where $\mathbb{S}_{\beta}(x)=\emptyset,$ $\sigma(\mathbb{S}_\beta(x))=0.$}
\begin{equation}\label{topological:condition}
 \sigma( \mathbb{S}_{\beta}(x))\leq \beta^{Q_0},
\end{equation} for some $Q_0\in \mathbb{R}$ such that  $0\leq Q_0<Q.$
\end{itemize}
\begin{itemize}
    \item[(ii)]  The group Fourier transform of $d\sigma$ satisfies the growth estimate
\begin{equation}\label{FT;Condition:Measure}
  \sup_{\pi\in\widehat{G}}\Vert\pi(\mathcal{R})^{ \frac{a}{\nu}}\widehat{d\sigma}(\pi)\Vert_{\textnormal{op}}<\infty,
\end{equation} for some order  $a>0.$
\item[(iii)] The following Tauberian  inequality
 \begin{equation}\label{Ft;Condition:derivative:measure}
        \sup_{s>0;\pi\in \widehat{G}}   \Vert  s\frac{d}{ds}\{\widehat{d\sigma}(s\cdot \pi)\}[(s\cdot \pi)(\mathcal{R})]^{\frac{a-\varepsilon_0}{\nu}}\Vert_{\textnormal{op}}<\infty,
    \end{equation} holds for some $\varepsilon_0\in \mathbb{R}$ satisfying that $a>\varepsilon_0\geq 0.$ 
\end{itemize} 
\end{definition}
We discuss the differentiability of the mapping $s\mapsto \frac{d}{ds}\{\widehat{d\sigma}(s\cdot \pi)\} ,$  in Subsection \ref{Differentiability:d:sigma}. Indeed, as expected one can define it as a densely defined operator on each representation space.

\subsection{Estimates for $M_{F}^{d\sigma}$ on graded Lie groups}\label{Subsection:mainresult} Next, we present the main result of this work, that is the extension to the (restricted weak  $(\frac{n}{n-1},1)$ inequality) $L^{\frac{n}{n-1},1}\rightarrow L^{\frac{n}{n-1},\infty}$ estimate by Bourgain \eqref{Bourgain:restricted} and the $L^p$-estimate of Stein \eqref{Stein:1976} to the setting of graded Lie groups.

\begin{theorem}\label{Maximal:Function:Graded} Let $d\sigma$ be a finite Borel measure of compact support on a graded Lie group $G$ of homogeneous dimension $Q.$ Assume that $d\sigma$ is $(Q,Q_0,\varepsilon_0,a)$-admissible (as in Definition \ref{Admissible:measure}).  Let $\mathcal{R}$ be a positive Rockland operator on $G$ of homogeneous degree $\nu>0.$ Let $c:=Q-Q_0>0,$ and let 
\begin{equation}\label{D:definition}
 D:=2a-\varepsilon_0+2c.   
\end{equation}
Then $M_{F}^{d\sigma}$ has the following continuity properties:
\begin{itemize}
    \item[(1)] $M_{F}^{d\sigma}: L^{\frac{D}{D-c},1}(G)\rightarrow L^{\frac{D}{D-c},\infty}(G) $ admits a bounded extension (equivalently, $M_{F}^{d\sigma}$ is of restricted weak  $(\frac{D}{D-c},1)$ type).
    \item[(2)] $M_{F}^{d\sigma}: L^{p}(G)\rightarrow L^{p}(G) $ is bounded for all $\frac{D}{D-c}<p\leq \infty.$
\end{itemize}    
\end{theorem}
Now we briefly discuss the previous result. Its proof is presented in Section \ref{Proof}.
\begin{remark}\label{remark:interpolation}
    In Theorem \ref{Maximal:Function:Graded}, the boundedness of $M_{F}^{d\sigma}: L^{\frac{D}{D-c},1}(G)\rightarrow L^{\frac{D}{D-c},\infty}(G) ,$ at the end-point $p=\frac{D}{D-c},$ is the main estimate.  Consequently, the $L^p$-estimate in (2) follows by interpolation of (1) with the boundedness of  $M_{F}^{d\sigma}:L^\infty(G)\rightarrow L^\infty(G).$ Moreover, by the Marcinkiewicz interpolation theorem (for sub-linear operators on Lorentz spaces) due to Hunt \cite[Page 804]{Hunt1964}\footnote{A linear version of this theorem was proved by Calder\'on in  \cite{Calderon1963}.}, we have that 
    $$  \boxed{ M_{F}^{d\sigma}: L^{p,q}(G)\rightarrow L^{p,s}(G) } $$ is bounded for all $\frac{D}{D-c}<p\leq \infty,$ and for all $1\leq q\leq s\leq \infty.$
\end{remark}
\begin{remark}\label{differentiation:th} Note that we have the inequality
\begin{equation}
    \Vert M_{F}^{d\sigma} \Vert_{L^p(G)}\leq C_{p,d\sigma}\Vert f\Vert_{L^p(G)},
\end{equation} for all $\frac{D}{D-c}<p\leq \infty,$
and the set 
$$\{f\in L^p(G): \lim_{t\rightarrow 0^{+}}\smallint\limits_{G}f\left(x\cdot (t\cdot y)^{-1}\right)d\sigma(y)=f(x),\,a.e.\}$$
is closed in $L^p(G),$ see e.g. Theorem 2.2 of Duoandikoetxea \cite[Page 27]{DuoandikoetxeaBook}. We can apply this very well-known remark to prove the differentiation criterion 
\begin{equation}\label{Eq:pointwise}
  \forall f\in L^p(G),\,  \lim_{t\rightarrow 0^{+}}\smallint\limits_{G}f\left(x\cdot (t\cdot y)^{-1}\right)d\sigma(y)=f(x),\,\textnormal{a. e.}
\end{equation}
In \eqref{Eq:pointwise} we have assumed that the measure $d\sigma$ is normalised $\sigma(G)=1.$   In some sense the previous result is a noncommutative version of the Stein differentiation theorem, see Section \ref{Motivation} for details. 
\end{remark}
\begin{remark} We observe that in the Euclidean setting the hypotheses in Theorem \ref{Maximal:Function:Graded} are satisfied e.g. in the case of the surface measure $d\sigma$ on the sphere. Indeed, when $G=\mathbb{R}^n,$ $Q=n,$ $Q_0=n-1,$ $c=1,$  $a=\frac{n-1}{2},$ and $\varepsilon_0=1.$ In consequence $D$ in \eqref{D:definition} agrees with the topological dimension $n.$ We then recover the classical range $\frac{n}{n-1}<p\leq \infty,$ in the $L^p$-boundendess of the  maximal function associated to the spherical averages.     
\end{remark}
\begin{remark} Consider the dyadic maximal function
\begin{equation}\label{Maximal:Function:Graded:D}
    {M}^{d\sigma}_Df(x)=\sup_{j\in \mathbb{Z}}\left|\smallint\limits_{G}f\left(x\cdot (2^j\cdot y)^{-1}\right)d\sigma(y)\right|.
\end{equation}
    In \cite{CardonaDelgadoRuzhanskyDyadic}, we have established a criterion based on the Fourier transform of the compactly supported measure $d\sigma,$ to guarantee the $L^p$-boundendess of ${M}^{d\sigma}_D.$ More precisely, if for some $a>0$ the Fourier transform of $d\sigma$ satisfies the growth estimate
\begin{equation}\label{FT;Condition:Measure:Dya}
\max_{\pm}  \sup_{\pi\in\widehat{G}}\Vert\pi(\mathcal{R})^{\pm \frac{a}{\nu}}\widehat{d\sigma}(\pi)\Vert_{\textnormal{op}}<\infty,
\end{equation}
then, the dyadic maximal function
${M}^{d\sigma}_D:L^p(G)\rightarrow L^p(G)$ can be extended to a bounded operator for all $1<p\leq \infty.$ This extension was motivated by the seminal work due to   C. P.   Calder\'on \cite{Calderon1979}. The proof in \cite{CardonaDelgadoRuzhanskyDyadic} is inspired by the approach developed by Duoandikoetxea and Rubio De Francia in \cite{Duoandikoetxea:RubioDeFrancia:86}.  
\end{remark}

\subsection{Some previous works}\label{References:section} 
In 1981 Cowling in \cite{Cowling81} proved that $p>\frac{Q}{Q-1},$ is a necessary and sufficient condition for the $L^p$-boundedness of the full maximal function (associated with the Koranyi sphere) on the Heisenberg group $\mathbb{H}_n.$ Here $Q=2n+2$ is the homogeneous dimension of $\mathbb{H}_n.$ In the setting of non-commutative harmonic analysis, after Cowling's result, 
there has been considerable interest in the mapping properties of the full maximal function $M_{F}^{d\sigma}$ and of its dyadic version $M_{D}^{d\sigma}.$ 

In the context of the Heisenberg group $\mathbb{H}_n$ and on two-steps nilpotent Lie groups, we refer to  Cowling \cite{Cowling79},   Schmidt \cite{Schmidt}, Fischer \cite{Fischer2006}, Ganguly and  Thangavelu \cite{GangulyThangavelu2021}, Lacey \cite{Lacey2019}, Nevo and Thangavelu \cite{NevoThangavelu}, Narayanan and  Thangavelu \cite{NarayananThangavelu2004},   M\"uller and Seeger \cite{MullerSeeger2004},   Bagchi, Hait, Roncal and Thangavelu \cite{BagchiHaitRoncalThangavelu2018}, and Beltr\'an, Guo, Hickman and Seeger in \cite{BeltranGuoHickmanSeeger}.

We also refer the reader to the recent work \cite{GovidanAswinHickman} due to Govidan Sheri,  Hickman, and  Wright,   where a {\it curvature assumption} was analysed to guarantee the $L^p$-boundedness of the dyadic maximal function on homogeneous nilpotent Lie groups. The estimate in \cite{GovidanAswinHickman} admits a substantial degree of generality recovering some previous results. We refer the reader to the introduction of \cite{GovidanAswinHickman} for details. 

Curiously, even in the Euclidean case, the problem of determining the weak (1,1) boundedness of the dyadic spherical means is open. For some results related to this problem, we refer to Stein \cite{ChristStein87},  Christ \cite{Christ88},   Seeger, Tao, and  Wright \cite{STW}, and the recent work by Cladek and Krause \cite{CladekKrause}. 

\subsection{Relation with the quantisation on graded Lie groups}

From the point of view of the microlocal analysis, in the setting of nilpotent Lie groups, when investigating the boundedness of the full maximal function, the problem of estimating the behaviour of the Fourier transform of the measure involves e.g. the stationary phase method or the Hadamard parametrix method. However, in view of the recent {\it quantisation theory on graded Lie groups} \cite{FischerRuzhanskyBook}, there is a research program dedicated to extending the techniques from the Euclidean harmonic analysis to the robust setting of nilpotent Lie groups, which can be traced back to the work Folland and Stein \cite{FollandStein1982} and Taylor \cite{Taylor1986}. 

According to this approach, one is expected to involve criteria in terms of the symbols of Rockland operators when formulating the boundedness of operators of relevant interest in harmonic analysis as (singular integrals like) Fourier multipliers, (strongly singular integral operators in the form of) pseudo-differential operators and Fourier integral operators, as well as maximal averages and their corresponding maximal functions.   

For criteria involving Rockland operators and other hypoelliptic operators we refer the reader to \cite{Cardona2017,CardonaRuzhansky:Besov:spaces:CRAS,CardonaRuzhansky2022:Triebel,CardonaRuzhansky:Besov:spaces,FischerRuzhansky2021,Guorong:Ruz,NursultanovRuzhanskyTikhonov,RuzhanskyWirth2015}, for H\"ormander type Mihlin theorems on several functions spaces, and to \cite{CardonaDelgadoRuzhansky:Lp:Pseudo,Cardona2022:Math:Z,Cardona2023:Math:Z,CardonaThesis,DelgadoRuzhansky2019} for the $L^p$-boundeness of pseudo-differential operators, Fourier integral operators and oscillating singular integrals. For the $L^p$-$L^q$-boundeness of Fourier multipliers and of pseudo-differential operators we refer to \cite{AkylzhanovRuzhansky2020,KumarRuzhansky2022,KumarRuzhansky2023,KumarRuzhansky2023:IMRN,David:Michael:lp:lq:2023}. Further remarks will be provided in Section \ref{FinalRemarks}.

\subsubsection{Structure of the paper} In Section \ref{preliminaries} we shall review the notation used in this work. We also shall present the main notions related to graded nilpotent Lie groups, its Fourier analysis, and the spectral calculus of Rockland operators. In Section \ref{Motivation} we shall motivate the class of measures used in our analysis of the full maximal function in the setting of graded Lie groups. Section \ref{Proof} will be dedicated to the proof of our main Theorem \ref{Maximal:Function:Graded}. Finally, in Section \ref{FinalRemarks} we present some conclusions of this work concerning some recent developments in the harmonic analysis on graded Lie groups, namely, the quantisation of operators on graded Lie groups as developed in \cite{FischerRuzhanskyBook}.

\section{Preliminaries}\label{preliminaries}

In this section we precise the notation used in our further analysis. We also present the notions related to the Fourier analysis of nilpotent Lie groups and the properties of Rockland operators relevant to this work. For a complete background on these subjects we refer to the references \cite{CorwinGreenleafBook,FischerRuzhanskyBook,FollandStein1982}.

\subsection{Notation} The Haar measure of a locally compact Hausdorff topological group $G$ is denoted by $dx.$  We recall that the $\sigma$-algebra generated by the open subsets of $G$ is called the Borel $\sigma$-algebra.
As usual, the Haar measure of a Borel set $A\subset G$ is denoted by $|A|,$ and the measure of $A$ with respect to an arbitrary Radon measure $d\sigma$ is denoted by $\sigma(A)=\smallint_Ad\sigma.$  

Here we always consider measures $d\sigma$ taking non-negative values  $\sigma(A)\geq 0$ (instead of signed or complex measures). We recall that the support $K$ of a measure $d\sigma$ is the smallest closed set $K$ such that $\sigma(G\setminus K)=0.$ 

Here the notations $A\asymp B$ and $A\sim B,$ both indicate that we have the inequality $C_1B\leq A\leq C_2 B,$ for some constants $C_1,C_2>0$ independent of fundamental parameters, while the notation $A\lesssim B$ indicates that we have the inequality $A\leq C_2 B$ with $C_2$ independent of fundamental parameters. 

Finally, for a bounded linear operator  $T:X\rightarrow Y$  between two Banach spaces $X,Y,$ the quantity $\Vert T\Vert_{\textnormal{op}}=\sup_{\Vert x\Vert=1}\Vert Tx\Vert$ denotes its standard operator norm.

\subsection{Nilpotent Lie groups} In this subsection we introduce some basic facts about nilpotent Lie groups. For a consistent treatment on the subject, we refer to Corwin and Greenleaf \cite{CorwinGreenleafBook}.

Let $\mathfrak{g}$ be a Lie algebra over $\mathbb{R}$ with Lie bracket $[\cdot,\cdot].$ A subset $\mathfrak{b}\subset \mathfrak{g}$ is an ideal of $ \mathfrak{g},$ if $[ \mathfrak{b}, \mathfrak{g}]\subset  \mathfrak{b},$ where 
\begin{equation}
    [ \mathfrak{b}, \mathfrak{g}]:=\{[B,X]:B\in\mathfrak{b},\,X\in \mathfrak{g}\}.
\end{equation}
The descending central series of $\mathfrak{g}$ is inductively defined by
\begin{equation*}
    \mathfrak{g}^{(1)}=\mathfrak{g},\quad \mathfrak{g}^{(n+1)}=[\mathfrak{g},\mathfrak{g}^{(n)}]=\mathbb{R}\textnormal{-span}\{[X,Y]:X\in \mathfrak{g},Y\in \mathfrak{g}^{(n)}\,\}.
\end{equation*}If $\phi:\mathfrak{g}\rightarrow \mathfrak{h}$ is a Lie algebra homomorphism, then $\phi(\mathfrak{g}^{(k)})\subset \mathfrak{h}^{(k)}.$ Note that, in view of the identity of Jacobi, we have that
for all $p,q\in \mathbb{N},$ $[\mathfrak{g}^{(p)},\mathfrak{g}^{(q)}]\subset \mathfrak{g}^{(p+q)}.$ In particular, each $\mathfrak{g}^{(k)}$ is an ideal of $\mathfrak{g},$ that is $[\mathfrak{g},\mathfrak{g}^{(k)}]\subset \mathfrak{g}^{(k)}.$

A Lie algebra $\mathfrak{g}$ is {\it nilpotent} if there is an integer $s$ in such a way that $\mathfrak{g}^{(s+1)}=\{0\}$  but $\mathfrak{g}^{(s)}\neq \emptyset.$ If this $s$ is minimal, we say that $\mathfrak{g}$ is an {\it $s$-step nilpotent Lie algebra}. Note that  $\mathfrak{g}$ is $s$-step nilpotent Lie algebra if and only if all brackets of at least $s+1$-elements of $\mathfrak{g}$ are zero but not all brackets of order $s$ are. Consequently, if $\mathfrak{g}$ is {\it nilpotent} (that is, $\mathfrak{g}$ is an $s$-nilpotent Lie algebra for some $s\in \mathbb{N}$), its center is not trivial. Indeed, if $\mathfrak{g}$ is $s$-step nilpotent, then $\mathfrak{g}^{(s)}$ is an abelian Lie algebra (also called central).

Now, let $\mathfrak{g}$ be any Lie algebra; let $\mathfrak{g}_{(1)}:=\textnormal{Center}(\mathfrak{g})$ be its center, and define
\begin{equation*}
    \mathfrak{g}_{(j)}:=\{X\in \mathfrak{g}: X\textnormal{ is central mod} \mathfrak{g}_{j-1} \}=\{X\in \mathfrak{g}: [X,\mathfrak{g}]\subset \mathfrak{g}_{(j-1)}\}.
\end{equation*}Each $\mathfrak{g}_{(j)}$ is an ideal. This sequence of ideals is called the {\it ascending central series of } $\mathfrak{g}.$ 
We have that a Lie algebra $\mathfrak{g}$ is $s$-nilpotent if and only if $\mathfrak{g}=\mathfrak{g}_{(s)}\neq \mathfrak{g}_{(s-1)}.$

A {\it nilpotent Lie group} $G$ is a Lie group whose Lie algebra $\mathfrak{g}\cong T_{e}G$ is nilpotent (here $e=e_{G}$ denotes the neutral element of $G$). If $G$ is a connected and simply connected nilpotent Lie group then we have that the exponential mapping
$$\textnormal{exp}:\mathfrak{g}\rightarrow G$$ is an analytic diffeomorphism. This diffeomorphism can be applied, among other things,  to reconstruct the group operation in view of the {\it Campbell-Baker-Hausdorff formula,} see \cite[Section 1.2]{CorwinGreenleafBook}.  
In the next subsection we are going to introduce the family of nilpotent Lie groups that will be considered here, namely, nilpotent homogeneous Lie groups and, in particular, graded Lie groups.

\subsection{Graded Lie groups $\&$ Rockland operators}
 Below we introduce some basic facts about the Fourier analysis and the spectral calculus of Rockland operators used in our further analysis. For the aspects of the theory of Rockland operators on graded Lie groups, we follow \cite{FischerRuzhanskyBook}.

$G$ denotes a homogeneous Lie group, that is a connected and simply connected nilpotent  Lie group whose Lie algebra $\mathfrak{g}$ is endowed with a family of dilations $D_{r,\mathfrak{g}}.$ We define it as follows. A family of dilations $  \textnormal{Dil}(\mathfrak{g}):=\{D_{r,\mathfrak{g}}:\,r>0\}$  on the Lie algebra $\mathfrak{g}$ is a family of  automorphisms on $\mathfrak{g}$  satisfying the following two compatibility conditions:
\begin{itemize}
\item[A.] For every $r>0,$ $D_{r,\mathfrak{g}}$ is a map of the form
$ D_{r,\mathfrak{g}}=\textnormal{Exp}(\ln(r)A), $ 
for some diagonalisable linear operator $A\equiv \textnormal{diag}[\nu_1,\cdots,\nu_n]:\mathfrak{g}\rightarrow \mathfrak{g}.$  
\item[B.] $\forall X,Y\in \mathfrak{g}, $ and $r>0,$ $[D_{r,\mathfrak{g}}X, D_{r,\mathfrak{g}}Y]=D_{r,\mathfrak{g}}[X,Y].$ 
\end{itemize} We call  the eigenvalues of the matrix $A,$ $\nu_1,\nu_2,\cdots,\nu_n,$ the dilations weights or weights of $G$. 

The homogeneous dimension of a homogeneous nilpotent Lie group $G$ whose dilations are defined via $ D_{r,\mathfrak{g}}=\textnormal{Exp}(\ln(r)A),$ is given by  $$  Q=\textnormal{Tr}(A)=\nu_1+\cdots+\nu_n,  $$  where $\nu_{i},$ $i=1,2,\cdots, n,$ are the eigenvalues of $A.$ The family of dilations $\textnormal{Dil}(\mathfrak{g})$ of the Lie algebra $\mathfrak{g}$ induces a family of  mappings on $G$ defined via,
$$\textnormal{Dil}(G):=\{ D_{r}:=\exp\circ D_{r,\mathfrak{g}} \circ \exp^{-1}:\,\, r>0\}. $$
We refer to the elements of this family as the dilations on the group. 

To simplify the notation we also denote $r\cdot x=D_{r}(x),$ $x\in G,$ $r>0.$  The dilations of the group and the Haar measure $dx$ on $G$ are related by the identity 
$$\smallint\limits_{G}f(x)dx= r^{Q}\smallint\limits_{G}(f\circ D_{r})(x)dx,\,\,\forall f\in L^1(G).$$

A connected, simply connected nilpotent Lie group $G$ is graded if its Lie algebra $\mathfrak{g}$ may be decomposed as the direct sum of subspaces $  \mathfrak{g}=\mathfrak{g}_{1}\oplus\mathfrak{g}_{2}\oplus \cdots \oplus \mathfrak{g}_{s}$  such that the following bracket conditions are satisfied: 
$[\mathfrak{g}_{i},\mathfrak{g}_{j} ]\subset \mathfrak{g}_{i+j},$ where  $ \mathfrak{g}_{i+j}=\{0\}$ if $i+j\geq s+1,$ for some $s.$ 

Examples of graded Lie groups are the Heisenberg group $\mathbb{H}_n$ and more generally any stratified group where the Lie algebra $ \mathfrak{g}$ is generated by the first stratum $\mathfrak{g}_{1}$. 

We say that $\pi$ is a continuous, unitary, and irreducible  representation of the group $G,$ if the following properties are satisfied,
\begin{itemize}
    \item[C.] $\pi\in \textnormal{Hom}(G, \textnormal{U}(H_{\pi})),$ for some separable Hilbert space $H_\pi,$ i.e. $\pi(xy)=\pi(x)\pi(y)$ and for the  adjoint of $\pi(x),$ $\pi(x)^*=\pi(x^{-1}),$ for every $x,y\in G.$ This property says that the representation is compatible with the group operation.
    \item[D.] The map $(x,v)\mapsto \pi(x)v, $ from $G\times H_\pi$ into $H_\pi$ is continuous. This says that the representation is a strongly continuous mapping.
    \item[E.] For every $x\in G,$ and $W_\pi\subset H_\pi,$ if $\pi(x)W_{\pi}\subset W_{\pi},$ then $W_\pi=H_\pi$ or $W_\pi=\{0\}.$ This means that  the representation $\pi$ is irreducible if its only invariant subspaces are $W=\{0\}$ and $W=H_\pi,$ the trivial ones. 
\end{itemize}
Two unitary representations $$ \pi\in \textnormal{Hom}(G,\textnormal{U}(H_\pi)) \textnormal{ and  }\eta\in \textnormal{Hom}(G,\textnormal{U}(H_\eta))$$  are equivalent if there exists a bounded invertible linear mapping $T:H_\pi\rightarrow H_\eta$ such that for any $x\in G,$ $T\pi(x)=\eta(x)T.$ 

The relation $\sim$ on the set of unitary and irreducible representations $\textnormal{Rep}(G)$ defined by: {\it $\pi\sim \eta$ if and only if $\pi$ and $\eta$ are equivalent representations,} is an equivalence relation. The quotient 
$$
    \widehat{G}:={\textnormal{Rep}(G)}/{\sim}
$$is called the unitary dual of $G.$ The Fourier transform takes values at elements of $\widehat{G},$ namely, the Fourier transform of $f\in L^1(G), $ at $\pi\in\widehat{G},$ is defined by 
\begin{equation*}
    \widehat{f}(\pi)\equiv [\mathscr{F}f](\pi):=\smallint\limits_{G}f(x)\pi(x)^*dx:H_\pi\rightarrow H_\pi,
\end{equation*} where the notation $\pi(x)^{*}$ denotes the adjoint operator of $\pi(x),$ $x\in G.$ 

In this work a central notion is the Fourier transform of finite measures. Let $d\sigma$ be a finite Borel measure on $G.$ We record that its Fourier transform at $\pi\in \widehat{G}$ is given by 
\begin{equation}
   \boxed{ \widehat{d\sigma}(\pi)=\smallint_G \pi(x)^{*}d\sigma(x)}
\end{equation}
The Fourier transform of a (positive) finite Borel measure is a bounded operator in any representation space, indeed
\begin{align*}
   \forall \pi \in \widehat{G},\, \Vert \widehat{d\sigma}(\pi)\Vert_{\textnormal{op}}\leq  \smallint_G \Vert\pi(x)^{*}\Vert_{\textnormal{op}}d\sigma(x)=\sigma(G)<\infty.
\end{align*}
The Schwartz space $\mathscr{S}(G)$ is defined by the smooth functions $f:G\rightarrow\mathbb{C},$ such that via the exponential mapping $f\circ \exp:\mathfrak{g}\cong \mathbb{R}^n\rightarrow \mathbb{C}$ can be identified with functions on the Schwartz class $\mathscr{S}(\mathbb{R}^n).$ Then, the Schwartz space on the unitary dual $\widehat{G}$ is defined by the image under the group Fourier transform of the Schwartz space $\mathscr{S}(G),$ that is  $\mathscr{F}:\mathscr{S}(G)\rightarrow \mathscr{S}(\widehat{G}):=\mathscr{F}(\mathscr{S}(G)).$
The Fourier inversion formula is given by \begin{equation*}
  \forall f\in L^1(G)\cap L^2(G),\,\, f(x)=\smallint\limits_{\widehat{G}}\textnormal{Tr}[\pi(x)\widehat{f}(\pi)]d\pi,\,\,x\in G.
\end{equation*} Here $d\pi$ denotes the Plancherel measure on $\widehat{G}.$ 
The $L^2$-space on the unitary dual $\widehat{G}$ is defined by the completion of the set $\mathscr{S}(G)$ with respect to the norm
\begin{equation}
    \Vert \sigma\Vert_{L^2(\widehat{G})}:=\left(\smallint\limits_{\widehat{G}}\|\sigma(\pi)\|_{\textnormal{HS}}^2d\pi \right)^{\frac{1}{2}},\,\,\sigma(\pi)=\widehat{f}(\pi)\in \mathscr{S}(\widehat{G}),
\end{equation}where $\|\cdot\|_{\textnormal{HS}}$ denotes the Hilbert-Schmidt norm of  operators on every representation space. The corresponding inner product on $L^2(\widehat{G})$ is given by
\begin{equation}
    (\sigma,\tau)_{L^2(\widehat{G})}:=\smallint\limits_{G}\textnormal{Tr}[\sigma(\pi)\tau(\pi)^{*}]d\pi,\,\,\sigma,\tau\in L^2(\widehat{G}).
\end{equation} Then,
the Plancherel theorem says that $\Vert f\Vert_{L^2(G)}=\Vert \widehat{f}\Vert_{L^2(\widehat{G})}$ for all $f\in L^2(G).$

A continuous linear operator $T:C^\infty(G)\rightarrow C^\infty(G)$ is homogeneous of  degree $\nu_T\in \mathbb{C}$ if for every $r>0$ the equality 
\begin{equation*}
T(f\circ D_{r})=r^{\nu_T}(Tf)\circ D_{r}
\end{equation*}
holds for every $f\in C^\infty(G). $

Let us consider $\pi\in\widehat{G},$ then $\pi:G\rightarrow U({H}_{\pi})$ is strongly continuous. We denote by ${H}_{\pi}^{\infty}$ the set of smooth vectors, also called G\r{a}rding vectors, formed by those $v\in {H}_{\pi}$ such that the function $x\mapsto \pi(x)v,$ $x\in \widehat{G},$ is smooth.  Then,  a Rockland operator is a left-invariant partial  differential operator $$ \mathcal{R}=\sum_{[\alpha]=\nu}a_{\alpha}X^{\alpha}:C^\infty(G)\rightarrow C^{\infty}(G),\quad [\alpha]:=\sum_{i=1}^{n}\alpha_{i}\nu_i,$$  which is homogeneous of positive degree $\nu=\nu_{\mathcal{R}}$ and such that, for every unitary irreducible non-trivial representation $\pi\in \widehat{G},$ its symbol $\pi(\mathcal{R})$ defined via the Fourier inversion formula by
\begin{equation}\label{Symbol:R}
   \mathcal{R}f(x)= \smallint\limits_{\widehat{G}}\textnormal{Tr}[\pi(x)\pi(\mathcal{R})\widehat{f}(\pi)]d\pi,\,\,x\in G,
\end{equation}
is injective on ${H}_{\pi}^{\infty}.$ Then, $\sigma_{\mathcal{R}}(\pi)=\pi(\mathcal{R})$ coincides with the infinitesimal representation of $\mathcal{R}$ as an element of the universal enveloping algebra $\mathfrak{U}(\mathfrak{g})$. 

\begin{example}
Let $G$ be a graded Lie group of topological dimension $n.$ We denote by $\{D_{r}\}_{r>0}$ the natural family of dilations of its Lie algebra $\mathfrak{g}:=\textnormal{Lie}(G),$ and by $\nu_1,\cdots,\nu_n$ its weights.  We fix a basis $Y=\{X_1,\cdots, X_{n}\}$ of  $\mathfrak{g}$ satisfying $D_{r}X_j=r^{\nu_j}X_{j},$ for $1\leq j\leq n,$ and all $r>0.$ If $\nu_{\circ}$ is any common multiple of $\nu_1,\cdots,\nu_n,$ the  operator 
$$ \mathcal{R}=\sum_{j=1}^{n}(-1)^{\frac{\nu_{\circ}}{\nu_j}}c_jX_{j}^{\frac{2\nu_{\circ}}{\nu_j}},\,\,c_j>0, $$ is a positive Rockland operator of homogeneous degree $2\nu_{\circ }$ on $G$ (see Lemma 4.1.8 of \cite{FischerRuzhanskyBook}).    
\end{example}
If the Rockland operator $\mathcal{R}$ is symmetric, then $\mathcal{R}$ and $\pi(\mathcal{R})$ admit self-adjoint extensions on $L^{2}(G)$ and ${H}_{\pi},$ respectively.
Now if we preserve the same notation for their self-adjoint
extensions and we denote by $E$ and $E_{\pi(\mathcal{R})}$  their spectral measures, we will denote by
\begin{equation}\label{Functional:identities:spectral:calculus}
    \psi(\mathcal{R})=\smallint\limits_{-\infty}^{\infty}\psi(\lambda) dE(\lambda),\,\,\,\textnormal{and}\,\,\,\pi(\psi(\mathcal{R}))\equiv \psi(\pi(\mathcal{R}))=\smallint\limits_{-\infty}^{\infty}\psi(\lambda) dE_{\pi(\mathcal{R})}(\lambda),
\end{equation}
the functions defined by the functional calculus. 

In general, we will reserve the notation $\{dE_A(\lambda)\}_{0\leq\lambda<\infty}$ for the spectral measure associated with a positive and self-adjoint operator $A$ on a Hilbert space $H.$ Moreover if $f\in L^{\infty}(\mathbb{R}^{+}_0),$ we are going to use the property
\begin{equation}\label{op:inequality}
    \boxed{\Vert f(A)\Vert_{\textnormal{op}}=\left\Vert \smallint_{0}^{\infty}f(\lambda)dE_A(\lambda) \right\Vert_{\textnormal{op}}\leq \Vert f\Vert_{L^{\infty}(\mathbb{R}^{+}_0)}}
\end{equation}
We now recall a lemma on dilations on the unitary dual $\widehat{G},$ which will be useful in our analysis of spectral multipliers.   For the proof, see Lemma 4.3 of \cite{FischerRuzhanskyBook}.
\begin{lemma}\label{dilationsrepre}
For every $\pi\in \widehat{G},$ let us define  
\begin{equation}\label{dilations:repre}
  D_{r}(\pi)(x)\equiv (r\cdot \pi)(x):=\pi(r\cdot x)\equiv \pi(D_r(x)),  
\end{equation}
 for every $r>0$ and all $x\in G.$ Then, if $f\in L^{\infty}(\mathbb{R}),$ $f(\pi^{(r)}(\mathcal{R}))=f({r^{\nu}\pi(\mathcal{R})}).$
\end{lemma}
We continuously shall use the notation 
\begin{equation}\label{start:notation}
    \boxed{K_{r}:=r^{-Q}K(r^{-1}\cdot)}
\end{equation} and the property 
\begin{equation}\label{Eq:dilatedFourier}
    \widehat{K}_{r}(\pi)=\smallint\limits_{G}r^{-Q}K(r^{-1}\cdot x)\pi(x)^*dx=\smallint\limits_{G}K(y)\pi(r\cdot y)^{*}dy=\widehat{K}(r\cdot \pi),
\end{equation}for $K\in L^1(G),$ any $\pi\in \widehat{G}$ and all $r>0,$ with $(r\cdot \pi)(y)=\pi(r\cdot y),$ $y\in G,$ as in \eqref{dilations:repre}.

The following lemma presents the action of the dilations of the group $G$ into the kernels of bounded functions of a Rockland operator $\mathcal{R},$ see \cite[Page 179]{FischerRuzhanskyBook}.
\begin{lemma}
Let $f\in L^{\infty}(\mathbb{R}^{+}_0)$ be a bounded Borel function and let $r>0.$ Then, we have
\begin{equation}\label{Fundamental:lemmaCZ:graded}
 \forall x\in G,\,   f(r^{\nu}\mathcal{R})\delta(x)=r^{-Q}[f(\mathcal{R})\delta](r^{-1}\cdot x),
\end{equation}where $Q$ is the homogeneous dimension of $G.$
\end{lemma}

\section{Motivating the curvature conditions}\label{Motivation}

Before presenting the proof of our main theorem in Section \ref{Proof}, we are going to motivate the curvature conditions (ii) and (iii) in Definition \ref{Admissible:measure}. To do so we follow the arguments presented in  Fefferman \cite[Page 527]{Fefferman2020:Stein:memory} describing the applications of the Tauberian type theorems, as observed by Stein in the proof of his differentiation theorem, see \cite[Theorem I.2.2]{Fefferman2020:Stein:memory}. We present our motivation in the next subsection.

\subsubsection{The Tauberian type condition}
Here, we are going to expose the arguments in \cite[Page 527]{Fefferman2020:Stein:memory} in the setting of homogeneous nilpotent Lie groups. First, we recall the Tauberian theorem:

\begin{itemize}
    \item {\it Suppose that $\lim_{R\rightarrow 0}\frac{1}{R}\smallint_0^RF(r)dr$ exists and that $\smallint_0^\infty r|\frac{dF}{dr}|^2dr<\infty.$ Then, $\lim_{r\rightarrow 0^{+}}F(r)$ exists, and equals $\lim_{R\rightarrow 0}\frac{1}{R}\smallint_0^RF(r)dr.$}
\end{itemize} Now, let us consider a compactly supported measure $d\sigma$ on $G,$ and let us identify the required conditions in order that 
\begin{equation}\label{Eq:pointwise:2}
  \forall f\in L^2(G),\quad  \lim_{t\rightarrow 0}\smallint\limits_{G}f\left(x\cdot (t\cdot y)^{-1}\right)d\sigma(y)=f(x),\,\textnormal{a. e.}
\end{equation}
In \eqref{Eq:pointwise:2} we have assumed that the measure $d\sigma$ is normalised $\sigma(G)=1.$    By the Riesz-Representation theorem we can write $d\sigma(y)=K(y)dy,$ where $dy$ denotes the Haar measure on $G.$ Then, using the notation in \eqref{start:notation}, the limit \eqref{Eq:pointwise:2} can be rewritten in terms of the convolution of the group as
\begin{equation}
     \forall f\in L^2(G),\,  \lim_{t\rightarrow 0}f\ast K_t(x)=f(x),\,\textnormal{a. e.}
\end{equation} 
To have control of the operator-valued group Fourier transform $\widehat{f}(\pi)$, we are going to use the following operator inequality,
\begin{equation}\label{Aux:1}
    |\textnormal{Tr}[f(\pi)^{*}\widehat{f}(\pi)B(\pi)]|\leq \Vert f(\pi)^{*}\widehat{f}(\pi)  \Vert_{S_{1}(H_\pi)}\Vert B(\pi)\Vert_{\textnormal{op}},
\end{equation}where $B(\pi)$ is any bounded operator on the representation space $H_\pi,$ $\Vert B(\pi)\Vert_{\textnormal{op}}$ is its operator norm, $S_{1}(H_\pi)$ denotes the ideal of trace class operators on $H_\pi,$ and $\Vert \cdot \Vert_{S_{1}(H_\pi)}$ is its norm which is given by the sum of the singular values of the operator.  Note that
\begin{equation}\label{Aux:2}
    \Vert f(\pi)^{*}\widehat{f}(\pi)  \Vert_{S_{1}(H_\pi)}=\Vert\widehat{f}(\pi)\Vert_{\textnormal{HS}}^2=\textnormal{Tr}[\widehat{f}(\pi)^{*}\widehat{f}(\pi)],
\end{equation} is the square of the Hilbert-Schmidt norm of $\widehat{f}(\pi).$

Let $f\in L^2(G).$ For each fixed $t>0,$ define  $F(x,t)$ by $\widehat{F}(\pi,t)=\widehat{d\sigma}(t\cdot \pi)\widehat{f}(\pi).$  
The differentiability of the mapping $t\mapsto \widehat{d\sigma}(t\cdot \pi) $  will be discussed in Subsection \ref{Differentiability:d:sigma} below.

Using Plancherel theorem we have that 
\begin{align*}
    \smallint_{G}\left( \smallint_0^\infty r\left| \frac{\partial}{\partial r}F(x,r)\right|^{2}dr  \right)dx &= \smallint_0^\infty \smallint_{G} r \left| \frac{\partial}{\partial r}F(x,r)\right|^{2} dxdr=  \smallint_0^\infty \smallint_{\widehat{G}} r \Vert \frac{\partial}{\partial r}\widehat{F}(\pi,r)\Vert_{\textnormal{HS}}^{2} d\pi dr\\
    &= \smallint_0^\infty \smallint_{\widehat{G}} r \Vert \frac{\partial}{\partial r}\left\{\widehat{d\sigma}(r\cdot \pi)\right\}\widehat{f}(\pi)\Vert_{\textnormal{HS}}^{2} d\pi dr\\
    &= \smallint_0^\infty \smallint_{\widehat{G}} r \textnormal{Tr}[\widehat{f}(\pi)^*\left|\frac{\partial}{\partial r}\left\{\widehat{d\sigma}(r\cdot \pi)\right\}\right|^2\widehat{f}(\pi)] d\pi dr,
\end{align*} where we have used the notation $|T|^2=T^*T,$ for the operator $T=\frac{\partial}{\partial r}\left\{\widehat{d\sigma}(r\cdot \pi)\right\}.$ 
Observe that 
\begin{align*}
  &\smallint_0^\infty \smallint_{\widehat{G}} r \textnormal{Tr}[\widehat{f}(\pi)^*\left|\frac{\partial}{\partial r}\left\{\widehat{d\sigma}(r\cdot \pi)\right\}\right|^2\widehat{f}(\pi)] d\pi dr\\
  &= \smallint_{\widehat{G}}  \textnormal{Tr}[\widehat{f}(\pi)\widehat{f}(\pi)^*\smallint_0^\infty r\left|\frac{\partial}{\partial r}\left\{\widehat{d\sigma}(r\cdot \pi)\right\}\right|^2 dr] d\pi  \\ 
  &\leq \smallint_{\widehat{G}}  \textnormal{Tr}[\widehat{f}(\pi)\widehat{f}(\pi)^*]\left\Vert\smallint_0^\infty r\left|\frac{\partial}{\partial r}\left\{\widehat{d\sigma}(r\cdot \pi)\right\}\right|^2 dr\right\Vert_{\textnormal{op}} d\pi\\
  &\leq \smallint_{\widehat{G}}  \textnormal{Tr}[\widehat{f}(\pi)\widehat{f}(\pi)^*] d\pi\sup_{\pi'\in \widehat{G}}  \smallint_0^\infty r\left\Vert\frac{\partial}{\partial r}\left\{\widehat{d\sigma}(r\cdot \pi')\right\}\right\Vert^2_{\textnormal{op}} dr.
\end{align*}
Now, if $d\sigma$ satisfies the integrability condition 
\begin{equation}
  \Sigma_{0}^2:= \sup_{\pi'\in \widehat{G}}  \smallint_0^\infty r\left\Vert\frac{\partial}{\partial r}\left\{\widehat{d\sigma}(r\cdot \pi')\right\}\right\Vert^2_{\textnormal{op}} dr<\infty,
\end{equation}  we have proved that 
\begin{equation}
    \smallint_{G}\left( \smallint_0^\infty r\left| \frac{\partial}{\partial r}F(x,r)\right|^{2}dr  \right)dx \leq \Sigma_0^2 \Vert f\Vert_{L^2}^2.
\end{equation}This analysis shows that 
\begin{equation}
    \smallint_0^\infty r\left| \frac{\partial}{\partial r}F(x,r)\right|^{2}dr <\infty,
\end{equation}for almost every $x\in G.$ On the other hand,  $F(x,r)=f\ast K_r(x)$ can be understood as the convolution of $f$ with a standard approximate identity, (see e.g.  Folland and Stein  \cite[Page 67]{FollandStein1982})  and $\lim_{R\rightarrow 0^{+}}\frac{1}{R}\smallint_0^R F(x,r)dr=f(x)$ for almost every $x\in G.$
In consequence, for almost every $x\in G,$
\begin{equation*}
    \lim_{r\rightarrow 0^{+}}F(x,r)=\lim_{R\rightarrow 0^{+}}\frac{1}{R}\smallint_0^RF(x,r)dr=f(x),
\end{equation*} proving the following version of the Stein differentiation theorem.
\begin{itemize}\item  {\it The equality:
$ \lim_{t\rightarrow 0}\smallint\limits_{G}f\left(x\cdot (t\cdot y)^{-1}\right)d\sigma(y)=f(x),$ holds almost everywhere for all $f\in L^2(G),$ provided that the normalised measure $d\sigma,$ being compactly supported,  satisfies the following Tauberian-type condition}
\begin{equation}\label{curvature:intuition}
     \sup_{\pi'\in \widehat{G}}  \smallint_0^\infty r\left\Vert\frac{\partial}{\partial r}\left\{\widehat{d\sigma}(r\cdot \pi')\right\}\right\Vert^2_{\textnormal{op}} dr<\infty.
\end{equation}
\end{itemize}
In the case of $\mathbb{R}^n$ the curvature of the sphere causes a decay of the Fourier transform $\widehat{d\sigma}$ at infinity allowing the convergence of the integral \eqref{curvature:intuition} when $n\geq 3$.

We note that if one wants to extend this criterion to other  $L^{p}(G)$ spaces one could impose conditions on the decay of the derivative $\frac{\partial}{\partial r}\{\widehat{d\sigma}(r\cdot \pi')\},$ indeed taking inspiration from the H\"ormander-Mihlin theorem on graded Lie groups, see Fischer and Ruzhansky \cite{FischerRuzhansky2021}.  With this kind of criteria in mind, the analysis above suggests that the condition (iii) in Definition \ref{Admissible:measure} could be a natural replacement for the curvature conditions available in the Euclidean setting. Since we use Rockland operators in this formulation the group $G$ is graded, see \cite[Page 172]{FischerRuzhanskyBook}. The hypothesis (ii) in  Definition \ref{Admissible:measure} is an analogy of the standard hypothesis  $|\widehat{d\sigma}(\xi)|\leq C|\xi|^{-a},$ $a>0,$ studied e.g. in the Euclidean setting in  \cite{Duoandikoetxea:RubioDeFrancia:86} by  Duoandikoetxea and  Rubio de Francia. Note that (i) in   Definition \ref{Admissible:measure} is a reasonable topological condition.

\subsubsection{About the differentiability of $t\mapsto \widehat{d\sigma}(t\cdot \pi)$}\label{Differentiability:d:sigma}
We finish this section by discussing the differentiability of the mapping $t\mapsto \widehat{d\sigma}(t\cdot \pi)v,$ on the subspace $ H_{\pi}^\infty$ of G\r{a}rding vectors $v.$ By following Kirillov \cite{Kirillov1962}, each representation space $H_{\pi}$ can be identified with $L^{2}(\mathbb{R}^{k(\pi)}),$ and also $H_{\pi}^\infty\cong \mathscr{S}(\mathbb{R}^{k(\pi)})$ can be identified with the Schwartz class on $\mathbb{R}^{k(\pi)}.$ Consequently, we deduce that $H_{\pi}^\infty,$ is a dense subspace of $H_{\pi}$ with respect to the topology induced by the norm in $H_{\pi}.$ We refer the reader to \cite[Page 42]{FischerRuzhansky2021} for a famous proof of this fact due to G\r{a}rding. 
 
Note that 
\begin{equation}\label{Derivative:discussion}
    \widehat{d\sigma}(t\cdot \pi)=\smallint_{G}\pi(t\cdot x)^{*}d\sigma(x)=\smallint_{G}\pi(t\cdot x)^{*}K(x)dx=\smallint_{G}\pi(t\cdot x)K(x^{-1})dx,
\end{equation}where we have written $d\sigma(x)=Kdx$ in view of the  Riesz-Representation theorem. Let us write \eqref{Derivative:discussion} in exponential coordinates $x=\exp(X),$ 
\begin{align*}
    \widehat{d\sigma}(t\cdot \pi) &=  \frac{d}{d t} \smallint_{\mathfrak{g}}\pi(t\cdot \exp(X))K((\exp{X})^{-1})|\det[D\exp(X)]|dX\\
    &=  \smallint_{\mathfrak{g}}\pi( \exp(t\cdot X))K((\exp{X})^{-1})|\det[D\exp(X)]|dX.
\end{align*} Since the group $G$ is graded, relative to the decomposition  $  \mathfrak{g}=\mathfrak{g}_{1}\oplus\mathfrak{g}_{2}\oplus \cdots \oplus \mathfrak{g}_{s},$ we have a unique representation $X=\sum_{i=1}^{s}\alpha_{i}(X)X_i,$ where $X_i\in \mathfrak{g}_{i}. $ and the $\alpha_{i}(X)$'s are scalars for $1\leq i\leq s.$  Note that 
$$ t\mapsto (t\cdot X)= \sum_{i=1}^{s}t^{\nu_i}\alpha_{i}(X)X_i=F_X(t)$$ is a $C^{\infty}$ mapping on $\mathbb{R}^+.$ We also recall that on the set of smooth vectors $H_\pi^{\infty},$ the mapping $x\mapsto \pi(x)v,$ is smooth for any $v\in H_{\pi}^{\infty}.$ With the notation above we have that
\begin{align*}
     \widehat{d\sigma}(t\cdot \pi) v= \smallint_{\mathfrak{g}}\pi( \exp[F_X(t)])vK((\exp{X})^{-1})|\det[D\exp(X)]|dX.
\end{align*} Since $ \pi( \exp[F_X(\cdot)])v$ is obtained by the composition $t\mapsto F_X(t)\mapsto \pi( \exp(F_X(t)))v,  $ it is a smooth mapping on $\mathbb{R}^+$. Note also that $K$ has compact support, allowing the mapping $\widehat{d\sigma}(t\cdot \pi) v$ to be smooth in $t>0.$ Then, by defining $\frac{d}{dt}\widehat{d\sigma}(t\cdot \pi):H_{\pi}^{\infty}\rightarrow H_{\pi},$ by $\frac{d}{dt}\widehat{d\sigma}(t\cdot \pi)v:= \frac{d}{dt}\{\widehat{d\sigma}(t\cdot \pi)v\},$ the operator $\frac{d}{dt}\widehat{d\sigma}(t\cdot \pi)$ can be realised as a densely defined operator on $H_\pi.$

\section{Proof of the main Theorem \ref{Maximal:Function:Graded}  }\label{Proof}

Let us prove that $M_{F}^{d\sigma}: L^{p_D,1}(G)\rightarrow L^{p_D,\infty}(G) $ admits a bounded extension, or  equivalently, that $M_{F}^{d\sigma}$ is of restricted weak  $(p_D,1)$ type for $p_D=\frac{D}{D-c}.$ This is the statement (1) of Theorem \ref{Maximal:Function:Graded}.  Then, interpolating the estimate  $M_{F}^{d\sigma}: L^{p_D,1}(G)\rightarrow L^{p_D,\infty}(G), $  with the $L^{\infty}(G)$-$L^{\infty}(G)$ boundedness of $M_{F}^{d\sigma},$  we get that $M_{F}^{d\sigma}: L^{p}(G)\rightarrow L^{p}(G) $ is bounded for all $p_D<p\leq \infty,$ (see Remark \ref{remark:interpolation}) which is the statement (2) of Theorem \ref{Maximal:Function:Graded}.

\subsection{Reduction of the problem}
 We are going to split the measure $d\sigma=d\sigma_1+d\sigma_2,$ with $\widehat{d\sigma}_1$ localising the low frequencies of the Fourier transform $\widehat{d\sigma}$ and with $\widehat{d\sigma}_2$ concentrating its high frequencies. 

To do this, let us consider a smooth function $\phi\geq 0$ supported on the interval $\{t\in \mathbb{R}:0\leq t\leq 2\}$ with $\phi(t)=1$ for $0\leq t\leq 1.$ For our further analysis, for a fixed $\beta>0,$ let us define the operator $\phi_\beta(\mathcal{R}^{1/\nu})$ by the functional calculus of Rockland operators, and consider its right convolution kernel
$\Phi_\beta= \phi_\beta(\mathcal{R}^{1/\nu})\delta.$ We denote $\Phi=\Phi_1,$ when $\beta=1$ for simplicity and note that in view of  \eqref{Fundamental:lemmaCZ:graded}, $\Phi_{\beta}=\beta^{-Q}\Phi(\beta^{-1}\cdot).$

Let us split the measure $d\sigma$ as follows 
\begin{equation}
   d\sigma= \Phi_\tau*d\sigma+[(\delta-\Phi_\tau)\ast d\sigma],\,d\sigma_{1}:=\Phi_\tau*d\sigma, 
\end{equation}   where the parameter $\tau,$ satisfying for now that $\tau^{-1}\gg 1$ will be chosen later. Let $f\in C^{\infty}_0(G).$
Let us consider the associated maximal functions for each measure in this partition, respectively,
\begin{equation}
    M_1 f(x):=\sup_{t>0}|f\ast (\Phi_\tau*d\sigma)_t (x)|,
\end{equation} and 
\begin{equation}
    M_2 f(x):=\sup_{t>0}|f\ast [(\delta-\Phi_\tau)\ast d\sigma]_t (x)|.
\end{equation}
In what follows, we shall investigate the $L^2$-theory of the operator $M_2$ and the $L^1$-theory of $M_1.$ 

The {\it curvature conditions} (ii) and (iii) in Definition \ref{Admissible:measure} will allow us to estimate the $L^2$-operator norm of $M_2.$ On the other hand, since the $L^1$-theory for $M_1$ is related to the growth properties of the measure $d\sigma$, in our further analysis we prove the weak (1,1) boundedness of $M_1$ using the property (i) in Definition \ref{Admissible:measure}. Once estimated the operator norms of $M_1$ and of $M_2,$ we establish the restricted weak estimate for $M_F^{d\sigma}$ in Subsection \ref{proof:subsection:full} which proves (1) in Theorem \ref{Maximal:Function:Graded}. 

\subsection{The $L^2$-theory for $M_2$} Now we are going to estimate the $L^2$-operator norm of $M_2.$ For this we require the {\it curvature conditions} \eqref{FT;Condition:Measure} and \eqref{Ft;Condition:derivative:measure}. We are going to use the inequality
$$ \sup_{t>0}\frac{1}{2}|F(x,t)|^2\leq \sup_{t>0}\left|\smallint_0^tF(x,s)F'(x,s)ds\right|\leq \smallint\limits_{0}^\infty |F(x,s)||F'(x,s)|ds $$
$$\leq \left( \smallint\limits_{0}^\infty |F(x,s)|^2\frac{ds}{s}\right)^{\frac{1}{2}}\left( \smallint\limits_{0}^\infty |sF'(x,s)|^2\frac{ds}{s}\right)^{\frac{1}{2}},$$
where $F'(x,s):=\frac{d}{ds}F(x,s),$ applied to 
$$F(x,s)=f\ast [(\delta-\Phi_\tau)\ast d\sigma]_t (x):=f\ast t^{-Q}[(\delta-\Phi_\tau)\ast d\sigma](t^{-1}\cdot x).$$ Observe that
$$  \sup_{t>0}\frac{1}{2}|F(x,t)|^2=\frac{1}{2}M_2f(x)^2.  $$
Consider the functions
\begin{equation}
    G_1f(x)= \left( \smallint\limits_{0}^\infty |F(x,s)|^2\frac{ds}{s}\right)^{\frac{1}{2}},
\end{equation}
and 
\begin{equation}
    G_2f(x)= \left( \smallint\limits_{0}^\infty |sF'(x,s)|^2\frac{ds}{s}\right)^{\frac{1}{2}}.
\end{equation}
Then
\begin{align*}
    \frac{1}{2}\smallint_G M_2f(x)^2dx\leq \smallint_G G_1(x)G_2(x)dx\leq \left(\smallint_G G_1f(x)^2dx\right)^{\frac{1}{2}}\left(\smallint_G G_2f(x)^2dx\right)^{\frac{1}{2}}=I_1\cdot I_2,
\end{align*}
where 
\begin{equation*}
    I_1:=\left(\smallint_G G_1f(x)^2dx\right)^{\frac{1}{2}},\,I_2:=\left(\smallint_G G_2f(x)^2dx\right)^{\frac{1}{2}}.
\end{equation*}

To estimate the $L^2$-norm of $M_2f,$ we are going to estimate the norms $I_1,$ and $I_2,$ respectively, in the following lemmas.

\begin{lemma}\label{Lemma:1} The Fourier transform condition
$$ \sup_{\pi'\in \widehat{G}}\Vert \widehat{d\sigma}(\pi')\pi'(\mathcal{R})^{\frac{a}{\nu}}\Vert_{\textnormal{op}}<\infty,$$
    implies that $I_1\lesssim \tau^{a}\Vert f\Vert_{L^2(G)}.$
\end{lemma}
Assuming that the symbol $s\frac{d}{ds}\{\widehat{d\sigma}(s\cdot \pi)\}$ decay with order $a-\varepsilon_0,$ we estimate $I_2$ as follows.

\begin{lemma}\label{Lemma:2} The Fourier transform condition
\begin{equation*}
\Sigma_{\varepsilon_0}:= \sup_{s>0;\pi\in \widehat{G}}   \Vert  s\frac{d}{ds}\{\widehat{d\sigma}(s\cdot \pi)\}[(s\cdot \pi)(\mathcal{R})]^{\frac{a-\varepsilon_0}{\nu}}\Vert_{\textnormal{op}}<\infty.
\end{equation*} implies that $I_1\lesssim \tau^{a-\varepsilon_0}\Vert f \Vert_{L^2(G)}.$    
\end{lemma}
Below we present the proof of these lemmas. Assuming these statements for a moment, note that  they imply the estimate
\begin{equation}
     \frac{1}{2}\smallint_G M_2f(x)^2dx\lesssim I_1 I_2\lesssim \tau^{a} \Vert f\Vert_{L^2(G)}\tau^{a-\varepsilon_0}\Vert f\Vert_{L^2(G)}=\tau^{2a-\varepsilon_0}\Vert f\Vert_{L^2(G)}^2.
\end{equation} In consequence we have the operator inequality
\begin{equation}\label{M:2:op:norm}
   \boxed{ \Vert  M_2f\Vert_{L^2(G)}\lesssim \tau^{\frac{2a-\varepsilon_0}{2}}\Vert f\Vert_{L^2(G)}}
\end{equation}
\begin{proof}[Proof of Lemma \ref{Lemma:1}]
Using Plancherel theorem, we have that 
\begin{align*}
    I_1^2:= &\smallint_G G_1f(x)^2dx=\smallint_G \smallint\limits_{0}^\infty |F(x,s)|^2\frac{ds}{s}dx= \smallint\limits_{0}^\infty \smallint_G|F(x,s)|^2 dx\frac{ds}{s}\\
    &=  \smallint\limits_{0}^\infty \smallint_{\widehat{G}}\Vert \widehat{F}(\pi,s)\Vert^2_{\textnormal{HS}} d\pi \frac{ds}{s}\\
    &=\smallint\limits_{0}^\infty \smallint_{\widehat{G}}\Vert \mathscr{F}[f\ast [(\delta-\Phi_\tau)\ast d\sigma]_t (\cdot)](\pi)\Vert^2_{\textnormal{HS}} d\pi \frac{ds}{s}\\
    &=\smallint\limits_{0}^\infty \smallint_{\widehat{G}}\Vert \mathscr{F}[ [(\delta-\Phi_\tau)\ast d\sigma]_t (\cdot)](\pi)\widehat{f}(\pi)\Vert^2_{\textnormal{HS}} d\pi \frac{ds}{s}\\
    &=\smallint\limits_{0}^\infty \smallint_{\widehat{G}}\Vert \mathscr{F}[ (\delta-\Phi_\tau)\ast d\sigma](s\cdot \pi)\widehat{f}(\pi)\Vert^2_{\textnormal{HS}} d\pi \frac{ds}{s}\\
    &=\smallint\limits_{0}^\infty \smallint_{\widehat{G}}\Vert \widehat{d\sigma}(s\cdot \pi) \mathscr{F}[ (\delta-\Phi_\tau)](s\cdot \pi)\widehat{f}(\pi)\Vert^2_{\textnormal{HS}} d\pi \frac{ds}{s}.
\end{align*}Now, using our Fourier transform condition
\begin{align*}
   \Sigma:= \sup_{\pi'\in \widehat{G}}\Vert \widehat{d\sigma}(\pi')\pi'(\mathcal{R})^{\frac{a}{\nu}}\Vert_{\textnormal{op}}<\infty,
\end{align*}with the variable $\pi'=s\cdot \pi$ given by the $s$-dilations of the representations, we have that 
\begin{align*}
    &I_1^2\\
    &= \smallint\limits_{0}^\infty \smallint_{\widehat{G}}\Vert \widehat{d\sigma}(s\cdot \pi) \mathscr{F}[ (\delta-\Phi_\tau)](s\cdot \pi)\widehat{f}(\pi)\Vert^2_{\textnormal{HS}} d\pi \frac{ds}{s}\\
    &= \smallint\limits_{0}^\infty \smallint_{\widehat{G}}\Vert \widehat{d\sigma}(s\cdot \pi) [(s\cdot \pi)(\mathcal{R})]^{\frac{a}{\nu}} [(s\cdot \pi)(\mathcal{R})]^{-\frac{a}{\nu}} \mathscr{F}[ (\delta-\Phi_\tau)](s\cdot \pi)\widehat{f}(\pi)\Vert^2_{\textnormal{HS}} d\pi \frac{ds}{s}\\
    &\leq  \smallint\limits_{0}^\infty \smallint_{\widehat{G}}\Vert \widehat{d\sigma}(s\cdot \pi) [(s\cdot \pi)(\mathcal{R})]^{\frac{a}{\nu}}\Vert_{\textnormal{op}}^2\Vert [(s\cdot \pi)(\mathcal{R})]^{-\frac{a}{\nu}} \mathscr{F}[ (\delta-\Phi_\tau)](s\cdot \pi)\widehat{f}(\pi)\Vert^2_{\textnormal{HS}} d\pi \frac{ds}{s}\\
    &\leq \Sigma^{2} \smallint\limits_{0}^\infty \smallint_{\widehat{G}}\Vert  [(s\cdot \pi)(\mathcal{R})]^{-\frac{a}{\nu}} \mathscr{F}[ (\delta-\Phi_\tau)](s\cdot \pi)\widehat{f}(\pi)\Vert^2_{\textnormal{HS}} d\pi \frac{ds}{s}.
\end{align*} Observe that
\begin{align*}
     &\smallint\limits_{0}^\infty \smallint_{\widehat{G}}\Vert  [(s\cdot \pi)(\mathcal{R})]^{-\frac{a}{\nu}} \mathscr{F}[ (\delta-\Phi_\tau)](s\cdot \pi)\widehat{f}(\pi)\Vert^2_{\textnormal{HS}} d\pi \frac{ds}{s}\\
     &= \smallint\limits_{0}^\infty \smallint_{\widehat{G}}\Vert  [(s\cdot \pi)(\mathcal{R})]^{-\frac{a}{\nu}}  [I_{H_{s\cdot \pi}}-\widehat{\Phi_\tau}(s\cdot \pi)]\widehat{f}(\pi)\Vert^2_{\textnormal{HS}} d\pi \frac{ds}{s}\\
     &= \smallint\limits_{0}^\infty \smallint_{\widehat{G}}\Vert  [(s\cdot \pi)(\mathcal{R})]^{-\frac{a}{\nu}}  [I_{H_{s\cdot \pi}}-\widehat{\Phi}(\tau \cdot s\cdot \pi)]\widehat{f}(\pi)\Vert^2_{\textnormal{HS}} d\pi \frac{ds}{s}\\
     &= \smallint\limits_{0}^\infty \smallint_{\widehat{G}}\Vert  [(s\cdot \pi)(\mathcal{R})]^{-\frac{a}{\nu}}  [I_{H_{s\cdot \pi}}-\widehat{\Phi}(\tau  s\cdot \pi)]\widehat{f}(\pi)\Vert^2_{\textnormal{HS}} d\pi \frac{ds}{s}\\
     &= \smallint\limits_{0}^\infty \smallint_{\widehat{G}}\Vert  [(s\cdot \pi)(\mathcal{R})]^{-\frac{a}{\nu}}  [I_{H_{ \pi}}-\Phi((\tau  s\cdot \pi)(\mathcal{R}^{\frac{1}{\nu}} ))]\widehat{f}(\pi)\Vert^2_{\textnormal{HS}} d\pi \frac{ds}{s}\\
     &= \smallint\limits_{0}^\infty \smallint_{\widehat{G}}\Vert  [(s\cdot \pi)(\mathcal{R})]^{-\frac{a}{\nu}}  [I_{H_{ \tau s\cdot \pi}}-\Phi((\tau  s\cdot \pi)(\mathcal{R}^{\frac{1}{\nu}} ))]\widehat{f}(\pi)\Vert^2_{\textnormal{HS}} d\pi \frac{ds}{s}\\
     &= \smallint\limits_{0}^\infty \smallint_{\widehat{G}}\Vert  [(s\cdot \pi)(\mathcal{R})]^{-\frac{a}{\nu}}  [(1-\Phi)((\tau  s\cdot \pi)(\mathcal{R}^{\frac{1}{\nu}} ))]\widehat{f}(\pi)\Vert^2_{\textnormal{HS}} d\pi \frac{ds}{s},
\end{align*}where we have used tha $I_{H_{ \pi}}=I_{H_{ s\cdot \pi}}=I_{H_{ \tau s\cdot \pi}},$ for any $s>0,$ since the  representations $s\cdot \pi=\pi(s\cdot )$ are acting on the same representation space $H_{\pi}.$ Observe that 
\begin{align*}
    &\smallint\limits_{0}^\infty \smallint_{\widehat{G}}\Vert  [(s\cdot \pi)(\mathcal{R})]^{-\frac{a}{\nu}}  [(1-\Phi)((\tau  s\cdot \pi)(\mathcal{R}^{\frac{1}{\nu}} ))]\widehat{f}(\pi)\Vert^2_{\textnormal{HS}} d\pi \frac{ds}{s}\\
    &= \smallint\limits_{0}^\infty \smallint_{\widehat{G}}\Vert  [s^\nu \pi(\mathcal{R})]^{-\frac{a}{\nu}}  [(1-\Phi)(\tau^\nu  s^\nu \pi(\mathcal{R} ))^{\frac{1}{\nu}}]\widehat{f}(\pi)\Vert^2_{\textnormal{HS}} d\pi \frac{ds}{s}\\
    &= \smallint\limits_{0}^\infty \smallint_{\widehat{G}}\Vert  s^{-a}[\pi(\mathcal{R})]^{-\frac{a}{\nu}}  [(1-\Phi)(\tau  s (\pi(\mathcal{R} ))^{\frac{1}{\nu}})]\widehat{f}(\pi)\Vert^2_{\textnormal{HS}} d\pi \frac{ds}{s}.
\end{align*} Now, we are going to use the functional calculus of Rockland operators. Let us proceed as follows, 
\begin{align*}
   & \smallint\limits_{0}^\infty \smallint_{\widehat{G}}\Vert  s^{-a}[\pi(\mathcal{R})]^{-\frac{a}{\nu}}  [(1-\Phi)(\tau  s (\pi(\mathcal{R} ))^{\frac{1}{\nu}})]\widehat{f}(\pi)\Vert^2_{\textnormal{HS}} d\pi \frac{ds}{s}\\
    &= \smallint\limits_{0}^\infty \smallint_{\widehat{G}}\Vert  s^{-a} \smallint_{0}^{\infty}\lambda^{-a}  [(1-\Phi)(\tau  s \lambda)]dE_{\pi(\mathcal{R})^{\frac{1}{\nu}}}\widehat{f}(\pi)\Vert^2_{\textnormal{HS}} d\pi \frac{ds}{s}\\
    &= \smallint\limits_{0}^\infty \smallint_{\widehat{G}}\Vert  s^{-a} \smallint_{0}^{\infty}\lambda^{-a}  [(1-\Phi)(\tau  s \lambda)]dE_{\pi(\mathcal{R})^{\frac{1}{\nu}}}(\lambda)\widehat{f}(\pi)\Vert^2_{\textnormal{HS}} d\pi \frac{ds}{s}.
\end{align*}
Denoting 
\begin{equation}
    A(s,\pi):= s^{-a} \smallint_{0}^{\infty}\lambda^{-a}  [(1-\Phi)(\tau  s \lambda)]dE_{\pi(\mathcal{R})^{\frac{1}{\nu}}}(\lambda)\widehat{f}(\pi)
\end{equation}we have that
\begin{equation}
    A(s,\pi)^*= \widehat{f}(\pi)^{*} s^{-a} \smallint_{0}^{\infty}\mu^{-a}  [(1-\Phi)(\tau  s \mu)]dE_{\pi(\mathcal{R})^{\frac{1}{\nu}}}(\mu).
\end{equation}Consequently,

\begin{align*}
    &\smallint\limits_{0}^\infty \smallint_{\widehat{G}}\Vert  s^{-a} \smallint_{0}^{\infty}\lambda^{-a}  [(1-\Phi)(\tau  s \lambda)]dE_{\pi(\mathcal{R})^{\frac{1}{\nu}}}(\lambda)\widehat{f}(\pi)\Vert^2_{\textnormal{HS}} d\pi \frac{ds}{s}\\
    &= \smallint\limits_{0}^\infty \smallint_{\widehat{G}}\Vert  A(s,\pi)\Vert^2_{\textnormal{HS}} d\pi \frac{ds}{s}\\
    &= \smallint\limits_{0}^\infty \smallint_{\widehat{G}}\textnormal{Tr}[  A(s,\pi)^*A(s,\pi)] d\pi \frac{ds}{s}\\
    &=\smallint\limits_{0}^\infty \smallint_{\widehat{G}}\textnormal{Tr}[\widehat{f}(\pi)^{*} s^{-a} \smallint_{0}^{\infty}\mu^{-a}  [(1-\Phi)(\tau  s \mu)]dE_{\pi(\mathcal{R})^{\frac{1}{\nu}}}(\mu)\\
    &\hspace{4cm} s^{-a} \smallint_{0}^{\infty}\lambda^{-a}  [(1-\Phi)(\tau  s \lambda)]dE_{\pi(\mathcal{R})^{\frac{1}{\nu}}}(\lambda)\widehat{f}(\pi)] d\pi \frac{ds}{s}\\
     &=\smallint\limits_{0}^\infty \smallint_{\widehat{G}}\textnormal{Tr}[\widehat{f}(\pi)\widehat{f}(\pi)^{*} s^{-a} \smallint_{0}^{\infty}\mu^{-a}  [(1-\Phi)(\tau  s \mu)]dE_{\pi(\mathcal{R})^{\frac{1}{\nu}}}(\mu)\\
    &\hspace{4cm} s^{-a} \smallint_{0}^{\infty}\lambda^{-a}  [(1-\Phi)(\tau  s \lambda)]dE_{\pi(\mathcal{R})^{\frac{1}{\nu}}}(\lambda)] d\pi \frac{ds}{s}\\
     &=\smallint\limits_{0}^\infty \smallint_{\widehat{G}}\textnormal{Tr}[\widehat{f}(\pi)\widehat{f}(\pi)^{*} s^{-2a} \smallint_{0}^{\infty}\lambda^{-2a}  [(1-\Phi)(\tau  s \lambda)]^2dE_{\pi(\mathcal{R})^{\frac{1}{\nu}}}(\lambda)] d\pi \frac{ds}{s}\\
     &= \smallint_{\widehat{G}}\textnormal{Tr}[\widehat{f}(\pi)\widehat{f}(\pi)^{*} \smallint\limits_{0}^\infty s^{-2a} \smallint_{0}^{\infty}\lambda^{-2a}  [(1-\Phi)(\tau  s \lambda)]^2dE_{\pi(\mathcal{R})^{\frac{1}{\nu}}}(\lambda)  \frac{ds}{s}  ] d\pi\\
      &= \smallint_{\widehat{G}}\textnormal{Tr}[\widehat{f}(\pi)\widehat{f}(\pi)^{*} \smallint\limits_{0}^\infty s^{-2a} \smallint_{\lambda\geq  \tau^{-1}s^{-1}}^{\infty}\lambda^{-2a}  [(1-\Phi)(\tau  s \lambda)]^2dE_{\pi(\mathcal{R})^{\frac{1}{\nu}}}(\lambda)  \frac{ds}{s}  ] d\pi,
\end{align*}where in the last integral we have used that $(1-\Phi)(\tau  s \lambda)=1-\Phi(\tau  s \lambda)=0,$ when $\lambda\leq  \tau^{-1}s^{-1}.$ Using Fubini's theorem we have that
\begin{align*}
    &\smallint_{\widehat{G}}\textnormal{Tr}[\widehat{f}(\pi)\widehat{f}(\pi)^{*} \smallint\limits_{0}^\infty s^{-2a} \smallint_{\lambda\geq  \tau^{-1}s^{-1}}^{\infty}\lambda^{-2a}  [(1-\Phi)(\tau  s \lambda)]^2dE_{\pi(\mathcal{R})^{\frac{1}{\nu}}}(\lambda)  \frac{ds}{s}  ] d\pi\\
    &=\smallint_{\widehat{G}}\textnormal{Tr}[\widehat{f}(\pi)\widehat{f}(\pi)^{*} \smallint\limits_{0}^\infty  \smallint_{\lambda\geq  \tau^{-1}s^{-1}}^{\infty} s^{-2a}\lambda^{-2a}  [(1-\Phi)(\tau  s \lambda)]^2dE_{\pi(\mathcal{R})^{\frac{1}{\nu}}}(\lambda)  \frac{ds}{s}  ] d\pi\\
    &=\smallint_{\widehat{G}}\textnormal{Tr}[\widehat{f}(\pi)\widehat{f}(\pi)^{*} \smallint\limits_{0}^\infty \smallint_{s\geq  \tau^{-1}\lambda^{-1}}^{\infty} s^{-2a} \lambda^{-2a}  [(1-\Phi)(\tau  s \lambda)]^2  \frac{ds}{s}  dE_{\pi(\mathcal{R})^{\frac{1}{\nu}}}(\lambda)] d\pi\\
    &\leq \smallint_{\widehat{G}}\textnormal{Tr}[\widehat{f}(\pi)\widehat{f}(\pi)^{*}]\left\Vert \smallint\limits_{0}^\infty \smallint_{s\geq  \tau^{-1}\lambda^{-1}}^{\infty} s^{-2a} \lambda^{-2a}  [(1-\Phi)(\tau  s \lambda)]^2  \frac{ds}{s} dE_{\pi(\mathcal{R})^{\frac{1}{\nu}}}(\lambda)  \right\Vert_{\textnormal{op}} d\pi.
\end{align*} Note that in the last inequality we have used \eqref{Aux:1} and \eqref{Aux:2}. Now, in order to estimate the operator norm inside the integral, let us apply \eqref{op:inequality} and  let us make the following substantial reduction
\begin{align*}
 & \left\Vert \smallint\limits_{0}^\infty \smallint_{s\geq  \tau^{-1}\lambda^{-1}}^{\infty} s^{-2a} \lambda^{-2a}  [(1-\Phi)(\tau  s \lambda)]^2 \frac{ds}{s}  dE_{\pi(\mathcal{R})^{\frac{1}{\nu}}}(\lambda) \right\Vert_{\textnormal{op}}\\
 &\leq \sup_{\lambda> 0} \smallint_{s\geq  \tau^{-1}\lambda^{-1}}^{\infty} s^{-2a} \lambda^{-2a}  [(1-\Phi)(\tau  s \lambda)]^2 \frac{ds}{s}\\
  &\leq \sup_{\lambda> 0} \smallint_{s\gg  \tau^{-1}\lambda^{-1}}^{\infty} s^{-2a} \lambda^{-2a}   \frac{ds}{s}= \sup_{\lambda> 0} \lambda^{-2a}  \smallint_{s\gg  \tau^{-1}\lambda^{-1}}^{\infty} s^{-2a-1}   ds\\
  &\sim  \sup_{\lambda> 0} \lambda^{-2a} ( -s^{-2a}/2a)|_{s\sim  \tau^{-1}\lambda^{-1} }^{\infty}\\
  &\sim  \sup_{\lambda> 0} \lambda^{-2a} \tau^{2a}\lambda^{2a}=\tau^{2a}.
\end{align*} Note that for the convergence of the previous integral we have used that $a>0.$ 
The analysis above implies that 
\begin{align*}
    I_1^2\lesssim \tau^{2a} \smallint_{\widehat{G}}\textnormal{Tr}[\widehat{f}(\pi)\widehat{f}(\pi)^{*}]d\pi=  \tau^{2a} \Vert f\Vert_{L^2(G)}^2.
\end{align*}
So, we have proved the estimate $I_1^2\lesssim \tau^{2a}$ which implies that
\begin{equation}
    I_1\lesssim \tau^a \Vert f\Vert_{L^2(G)}.
\end{equation} The proof of Lemma \ref{Lemma:1} is complete.
\end{proof}

\begin{proof}[Proof of Lemma \ref{Lemma:2}]
As for the estimate of the $L^2$-norm  of 
$$
    G_2f(x)= \left( \smallint\limits_{0}^\infty |sF'(x,s)|^2\frac{ds}{s}\right)^{\frac{1}{2}},
$$ let us apply Plancherel theorem as follows
\begin{align*}
    I_2^2:= &\smallint_G G_2f(x)^2dx=\smallint_G \smallint\limits_{0}^\infty |sF'(x,s)|^2\frac{ds}{s}dx= \smallint\limits_{0}^\infty \smallint_G|sF'(x,s)|^2 dx\frac{ds}{s}\\
    &=  \smallint\limits_{0}^\infty \smallint_{\widehat{G}}\Vert s \widehat{F}'(\pi,s)\Vert^2_{\textnormal{HS}} d\pi \frac{ds}{s},
\end{align*}where $\widehat{F}'(\pi,s)$ denotes the derivative with respect to $s$ of $s\mapsto F(\pi,s).$  Since
$$ \widehat{F}(\pi,s)= \widehat{d\sigma}(s\cdot \pi)(I_{H_\pi}-\widehat{\Phi}_\tau(s\cdot \pi))\widehat{f}(\pi), $$ we have that
\begin{align*}
    \widehat{F}'(\pi,s)=[\frac{d}{ds}\{\widehat{d\sigma}(s\cdot \pi)\}(I_{H_\pi}-\widehat{\Phi}_\tau(s\cdot \pi))+\widehat{d\sigma}(s\cdot \pi)\frac{d}{ds}\{(I_{H_\pi}-\widehat{\Phi}_\tau(s\cdot \pi))\}]\widehat{f}(\pi).
\end{align*}Now, observe that
\begin{align*}
    I_{2}&=\left(\smallint\limits_{0}^\infty \smallint_{\widehat{G}}\Vert s \widehat{F}'(\pi,s)\Vert^2_{\textnormal{HS}} d\pi \frac{ds}{s}\right)^{\frac{1}{2}}\\
    &\leq  \left(\smallint\limits_{0}^\infty \smallint_{\widehat{G}}\Vert s[\frac{d}{ds}\{\widehat{d\sigma}(s\cdot \pi)\}(I_{H_\pi}-\widehat{\Phi}_\tau(s\cdot \pi))]\widehat{f}(\pi)\Vert^2_{\textnormal{HS}} d\pi \frac{ds}{s}\right)^{\frac{1}{2}}\\
    &+  \left(\smallint\limits_{0}^\infty \smallint_{\widehat{G}}\Vert s[\widehat{d\sigma}(s\cdot \pi)\frac{d}{ds}\{(I_{H_\pi}-\widehat{\Phi}_\tau(s\cdot \pi))\}]\widehat{f}(\pi)\Vert^2_{\textnormal{HS}} d\pi \frac{ds}{s}\right)^{\frac{1}{2}}\\
    &:=I_{2,1}+I_{2,2}.
\end{align*} To estimate 
\begin{equation}
    I_{2,1}:=  \left(\smallint\limits_{0}^\infty \smallint_{\widehat{G}}\Vert s[\frac{d}{ds}\{\widehat{d\sigma}(s\cdot \pi)\}(I_{H_\pi}-\widehat{\Phi}_\tau(s\cdot \pi))]\widehat{f}(\pi)\Vert^2_{\textnormal{HS}} d\pi \frac{ds}{s}\right)^{\frac{1}{2}},
\end{equation} consider $\varepsilon_0>0$ such that the derivative  $$\frac{d}{ds}\{\widehat{d\sigma}(s\cdot \pi)\}(I_{H_\pi}-\widehat{\Phi}_\tau(s\cdot \pi))$$ satisfies the Fourier transform condition
\begin{equation}
\Sigma_{\varepsilon_0}:= \sup_{s>0;\pi\in \widehat{G}}   \Vert  s\frac{d}{ds}\{\widehat{d\sigma}(s\cdot \pi)\}[(s\cdot \pi)(\mathcal{R})]^{\frac{a-\varepsilon_0}{\nu}}\Vert_{\textnormal{op}}<\infty.
\end{equation} Then, we can use the argument above in the estimate of $I_1,$ in order to estimate the term $I_{2,1}.$ We repeat the argument since our interpolation technique will be sensible to a suitable condition involving the parameter $\varepsilon_0.$ For now, we assume that $a> \varepsilon_0;$ this condition  will be clarified later. Indeed, observe that 
\begin{align*}
     I_{2,1}^2 &\leq  \Sigma^{2}_{\varepsilon_0} \smallint\limits_{0}^\infty \smallint_{\widehat{G}}\Vert   [(s\cdot \pi)(\mathcal{R})]^{-\frac{a-\varepsilon_0}{\nu}} \mathscr{F}[ (\delta-\Phi_\tau)](s\cdot \pi)\widehat{f}(\pi)\Vert^2_{\textnormal{HS}} d\pi \frac{ds}{s}\\
     &=\Sigma_{\varepsilon_0}^2  \smallint\limits_{0}^\infty \smallint_{\widehat{G}}\Vert   [(s\cdot \pi)(\mathcal{R})]^{-\frac{a-\varepsilon_0}{\nu}}  [(1-\Phi)((\tau  s\cdot \pi)(\mathcal{R}^{\frac{1}{\nu}} ))]\widehat{f}(\pi)\Vert^2_{\textnormal{HS}} d\pi \frac{ds}{s}\\
     &=\Sigma_{\varepsilon_0}^2\smallint\limits_{0}^\infty \smallint_{\widehat{G}}\Vert  s^{-(a-\varepsilon_0)}[\pi(\mathcal{R})]^{-\frac{a-\varepsilon_0}{\nu}}  [(1-\Phi)(\tau  s (\pi(\mathcal{R} ))^{\frac{1}{\nu}})]\widehat{f}(\pi)\Vert^2_{\textnormal{HS}} d\pi \frac{ds}{s}.
\end{align*}Denoting 
\begin{equation}
    \tilde A(s,\pi)= s^{-(a-\varepsilon_0)} \smallint_{0}^{\infty}\lambda^{-(a-\varepsilon_0)}  [(1-\Phi)(\tau  s \lambda)]dE_{\pi(\mathcal{R})^{\frac{1}{\nu}}}(\lambda)\widehat{f}(\pi)
\end{equation}we have that
\begin{equation}
    \tilde A(s,\pi)^*= \widehat{f}(\pi)^{*} s^{-(a-\varepsilon_0)} \smallint_{0}^{\infty}\mu^{-(a-\varepsilon_0)}  [(1-\Phi)(\tau  s \mu)]dE_{\pi(\mathcal{R})^{\frac{1}{\nu}}}(\mu).
\end{equation}Consequently,
\begin{align*}
    &\smallint\limits_{0}^\infty \smallint_{\widehat{G}}\Vert  s^{-(a-\varepsilon_0)}[\pi(\mathcal{R})]^{-\frac{a-\varepsilon_0}{\nu}}  [(1-\Phi)(\tau  s (\pi(\mathcal{R} ))^{\frac{1}{\nu}})]\widehat{f}(\pi)\Vert^2_{\textnormal{HS}} d\pi \frac{ds}{s}\\
   &= \smallint\limits_{0}^\infty \smallint_{\widehat{G}}\Vert  \tilde A(s,\pi)\Vert^2_{\textnormal{HS}} d\pi \frac{ds}{s}\\
    &= \smallint\limits_{0}^\infty \smallint_{\widehat{G}}\textnormal{Tr}[  \tilde A(s,\pi)^*\tilde A(s,\pi)] d\pi \frac{ds}{s}\\
    &=\smallint_{\widehat{G}}\textnormal{Tr}[\widehat{f}(\pi)\widehat{f}(\pi)^{*} \smallint\limits_{0}^\infty \smallint_{s\geq  \tau^{-1}\lambda^{-1}}^{\infty} s^{-2(a-\varepsilon_0)} \lambda^{-2(a-\varepsilon_0)}  [(1-\Phi)(\tau  s \lambda)]^2  \frac{ds}{s}  dE_{\pi(\mathcal{R})^{\frac{1}{\nu}}}(\lambda)] d\pi\\
    &\leq \smallint_{\widehat{G}}\textnormal{Tr}[\widehat{f}(\pi)\widehat{f}(\pi)^{*}]\left\Vert \smallint\limits_{0}^\infty \smallint_{s\geq  \tau^{-1}\lambda^{-1}}^{\infty} s^{-2(a-\varepsilon_0)} \lambda^{-2(a-\varepsilon_0)}  [(1-\Phi)(\tau  s \lambda)]^2  \frac{ds}{s} dE_{\pi(\mathcal{R})^{\frac{1}{\nu}}}(\lambda)  \right\Vert_{\textnormal{op}} d\pi.
\end{align*} Now, to estimate the operator norm inside the integral, observe that in view of \eqref{op:inequality} we have that
\begin{align*}
 & \left\Vert \smallint\limits_{0}^\infty \smallint_{s\geq  \tau^{-1}\lambda^{-1}}^{\infty} s^{-2(a-\varepsilon_0)} \lambda^{-2(a-\varepsilon_0)}  [(1-\Phi)(\tau  s \lambda)]^2 \frac{ds}{s}  dE_{\pi(\mathcal{R})^{\frac{1}{\nu}}}(\lambda) \right\Vert_{\textnormal{op}}\\
 &\leq \sup_{\lambda> 0} \smallint_{s\geq  \tau^{-1}\lambda^{-1}}^{\infty} s^{-2(a-\varepsilon_0)} \lambda^{-2(a-\varepsilon_0)}  [(1-\Phi)(\tau  s \lambda)]^2 \frac{ds}{s}\\
  &\leq \sup_{\lambda> 0} \smallint_{s\gg  \tau^{-1}\lambda^{-1}}^{\infty} s^{-2(a-\varepsilon_0)} \lambda^{-2(a-\varepsilon_0)}   \frac{ds}{s}= \sup_{\lambda> 0} \lambda^{-2(a-\varepsilon_0)}  \smallint_{s\gg  \tau^{-1}\lambda^{-1}}^{\infty} s^{-2(a-\varepsilon_0)-1}   ds\\
  &\sim  \sup_{\lambda> 0} \lambda^{-2a} ( -s^{-2(a-\varepsilon_0)}/2(a-\varepsilon_0))|_{s\sim  \tau^{-1}\lambda^{-1} }^{\infty}\\
  &\sim  \sup_{\lambda> 0} \lambda^{-2(a-\varepsilon_0)} \tau^{2(a-\varepsilon_0)}\lambda^{2(a-\varepsilon_0)}=\tau^{2(a-\varepsilon_0)},
\end{align*}where we have used the condition $a>\varepsilon_0$ in the convergence of the integral above.  All the analyisis above implies that 
\begin{align*}
    I_{2,1}\lesssim \tau^{a-\varepsilon_0} \Vert f\Vert_{L^2(G)}.
\end{align*}Indeed, note that 
\begin{align*}
     I_{2,1}^2\lesssim  \tau^{2(a-\varepsilon_0)} \smallint_{\widehat{G}}\textnormal{Tr}[\widehat{f}(\pi)\widehat{f}(\pi)^{*}] d\pi=  \tau^{2(a-\varepsilon_0)} \Vert f\Vert_{L^2(G)}^2.
\end{align*}
We claim that the symbol
$$ \widehat{d\sigma}(s\cdot \pi)\frac{d}{ds}\{(I_{H_\pi}-\widehat{\Phi}_\tau(s\cdot \pi))\}]\widehat{f}(\pi)$$
is ``smoothing" and consequently the contribution of the integral
\begin{align*}
    I_{2,2}=  \left(\smallint\limits_{0}^\infty \smallint_{\widehat{G}}\Vert s[\widehat{d\sigma}(s\cdot \pi)\frac{d}{ds}\{(I_{H_\pi}-\widehat{\Phi}_\tau(s\cdot \pi))\}]\widehat{f}(\pi)\Vert^2_{\textnormal{HS}} d\pi \frac{ds}{s}\right)^{\frac{1}{2}}   
\end{align*} is less significant than the one by $I_{2,2}.$ Indeed,  we are going to prove that 
\begin{equation}\label{error:term}
    I_{2,2}\lesssim I_{2,1}.
\end{equation}  In order to estimate 
\begin{align*}
    I_{2,2}=  \left(\smallint\limits_{0}^\infty \smallint_{\widehat{G}}\Vert s[\widehat{d\sigma}(s\cdot \pi)\frac{d}{ds}\{(I_{H_\pi}-\widehat{\Phi}_\tau(s\cdot \pi))\}]\widehat{f}(\pi)\Vert^2_{\textnormal{HS}} d\pi \frac{ds}{s}\right)^{\frac{1}{2}},
\end{align*} let us explicitly compute the term
\begin{align*}
    \frac{d}{ds}\{(I_{H_\pi}-\widehat{\Phi}_\tau(s\cdot \pi))\} &= \frac{d}{ds}\{(I_{H_\pi}-{\Phi}(\tau s\cdot \pi(\mathcal{R})^{\frac{1}{\nu}}))\}.
\end{align*}Indeed, observe that 
\begin{align*}
     \frac{d}{ds}\{(I_{H_\pi}-\widehat{\Phi}_\tau(s\cdot \pi))\} &=  \frac{d}{ds}\smallint\limits_0^\infty(1-\phi(s\tau\lambda))dE_{\pi(\mathcal{R})^{\frac{1}{\nu}}}(\lambda)\\
     &=\smallint\limits_0^\infty-\phi'(s\tau\lambda)\tau\lambda dE_{\pi(\mathcal{R})^{\frac{1}{\nu}}}(\lambda)\\
     &=-\phi'(s\tau\cdot \pi(\mathcal{R})^{\frac{1}{\nu}})\pi(\mathcal{R})^{\frac{1}{\nu}}\tau.
\end{align*} Consequently,
\begin{align*}
    I_{2,2}^2&= \smallint\limits_{0}^\infty \smallint_{\widehat{G}}\Vert s[\widehat{d\sigma}(s\cdot \pi)\frac{d}{ds}\{(I_{H_\pi}-\widehat{\Phi}_\tau(s\cdot \pi))\}]\widehat{f}(\pi)\Vert^2_{\textnormal{HS}} d\pi \frac{ds}{s}\\
    &= \smallint\limits_{0}^\infty \smallint_{\widehat{G}}\Vert s[\widehat{d\sigma}(s\cdot \pi) \phi'(s\tau\cdot \pi(\mathcal{R})^{\frac{1}{\nu}})\pi(\mathcal{R})^{\frac{1}{\nu}}\tau ]\widehat{f}(\pi)\Vert^2_{\textnormal{HS}} d\pi \frac{ds}{s}\\
     &= \smallint\limits_{0}^\infty \smallint_{\widehat{G}}\Vert [\widehat{d\sigma}(s\cdot \pi)[(s\cdot \pi)(\mathcal{R})]^{\frac{a}{\nu}} [(s\cdot \pi)(\mathcal{R})]^{-\frac{a}{\nu}}\\
     &\hspace{2cm}\phi'(s\tau\cdot \pi(\mathcal{R})^{\frac{1}{\nu}})\pi(\mathcal{R})^{\frac{1}{\nu}}s\tau ]\widehat{f}(\pi)\Vert^2_{\textnormal{HS}} d\pi \frac{ds}{s}\\
     &\leq \Sigma^2 \smallint\limits_{0}^\infty \smallint_{\widehat{G}}\Vert [ [(s\cdot \pi)(\mathcal{R})]^{-\frac{a}{\nu}}\phi'(s\tau\cdot \pi(\mathcal{R})^{\frac{1}{\nu}})\pi(\mathcal{R})^{\frac{1}{\nu}}s\tau ]\widehat{f}(\pi)\Vert^2_{\textnormal{HS}} d\pi \frac{ds}{s}.
\end{align*}
By denoting 
\begin{equation}
B(s,\pi)= [ [(s\cdot \pi)(\mathcal{R})]^{-\frac{a}{\nu}}\phi'(s\tau\cdot \pi(\mathcal{R})^{\frac{1}{\nu}})\pi(\mathcal{R})^{\frac{1}{\nu}}s\tau ]\widehat{f}(\pi),
\end{equation} we have that 
\begin{equation}
    B(s,\pi)^{*}= \widehat{f}(\pi)^*[ [(s\cdot \pi)(\mathcal{R})]^{-\frac{a}{\nu}}\phi'(s\tau\cdot \pi(\mathcal{R})^{\frac{1}{\nu}})\pi(\mathcal{R})^{\frac{1}{\nu}}s\tau ].
\end{equation} Consequently 
\begin{align*}
    & \smallint\limits_{0}^\infty \smallint_{\widehat{G}}\Vert [ [(s\cdot \pi)(\mathcal{R})]^{-\frac{a}{\nu}}\phi'(s\tau\cdot \pi(\mathcal{R})^{\frac{1}{\nu}})\pi(\mathcal{R})^{\frac{1}{\nu}}s\tau ]\widehat{f}(\pi)\Vert^2_{\textnormal{HS}} d\pi \frac{ds}{s}\\
    &= \smallint\limits_{0}^\infty \smallint_{\widehat{G}}\textnormal{Tr}[ 
        B(s,\pi)^{*}B(s,\pi)  ] d\pi \frac{ds}{s}\\
    &= \smallint\limits_{0}^\infty \smallint_{\widehat{G}}\textnormal{Tr}[   \widehat{f}(\pi) \widehat{f}(\pi)^*
        [ [(s\cdot \pi)(\mathcal{R})]^{-\frac{2a}{\nu}}\phi'(s\tau\cdot \pi(\mathcal{R})^{\frac{1}{\nu}})^2\pi(\mathcal{R})^{\frac{2}{\nu}}s\tau ] ] d\pi \frac{ds}{s}    \\
    &=  \smallint_{\widehat{G}}\textnormal{Tr}[   \widehat{f}(\pi) \widehat{f}(\pi)^*
       \smallint\limits_{0}^\infty [ [(s\cdot \pi)(\mathcal{R})]^{-\frac{2a}{\nu}}\phi'(s\tau\cdot \pi(\mathcal{R})^{\frac{1}{\nu}})^{2}\pi(\mathcal{R})^{\frac{2}{\nu}}s^2\tau^2 ]\frac{ds}{s} ] d\pi   \\
    &\leq \smallint_{\widehat{G}}\textnormal{Tr}[   \widehat{f}(\pi) \widehat{f}(\pi)^*
        ]\left\Vert \smallint\limits_{0}^\infty [ [(s\cdot \pi)(\mathcal{R})]^{-\frac{2a}{\nu}}\phi'(s\tau\cdot \pi(\mathcal{R})^{\frac{1}{\nu}})^{2}\pi(\mathcal{R})^{\frac{2}{\nu}}s^2\tau^2 ]\frac{ds}{s}\right\Vert_{\textnormal{op}} d\pi.    
\end{align*}In order to estimate the operator norm inside the integral, note that the support of $\phi'$ is contained in the interval $[1,2].$ So using the property in \eqref{op:inequality} we proceed as follows,
\begin{align*}
   & \left\Vert \smallint\limits_{0}^\infty [ [(s\cdot \pi)(\mathcal{R})]^{-\frac{2a}{\nu}}\phi'(s\tau\cdot \pi(\mathcal{R})^{\frac{1}{\nu}})^{2}\pi(\mathcal{R})^{\frac{2}{\nu}}s^2\tau^2 ]\frac{ds}{s}\right\Vert_{\textnormal{op}}\\
    &= \left\Vert \smallint\limits_{0}^\infty [ [(s\cdot \pi)(\mathcal{R})]^{-\frac{2a}{\nu}}\phi'(s\tau\cdot \pi(\mathcal{R})^{\frac{1}{\nu}})^{2}\pi(\mathcal{R})^{\frac{2}{\nu}}s^2\tau^2 ]\frac{ds}{s}\right\Vert_{\textnormal{op}}\\
    &= \left\Vert \smallint\limits_{0}^\infty \smallint\limits_0^\infty s^{-2a}\lambda^{-2a}\phi'(s\tau\lambda)^2\lambda^2s^2\tau^2 dE_{\pi(\mathcal{R})^{\frac{1}{\nu}}} (\lambda)\frac{ds}{s}\right\Vert_{\textnormal{op}}\\
    &= \left\Vert \smallint\limits_{0}^\infty \smallint\limits_0^\infty s^{-2a}\lambda^{-2a}\phi'(s\tau\lambda)^2\lambda^2s^2\tau^2  \frac{ds}{s} dE_{\pi(\mathcal{R})^{\frac{1}{\nu}}}(\lambda) \right\Vert_{\textnormal{op}}\\
    &\leq \sup_{\lambda>0}|\smallint\limits_0^\infty s^{-2a}\lambda^{-2a}\phi'(s\tau\lambda)^2\lambda^2s^2\tau^2  \frac{ds}{s}|\\
    &\leq \sup_{\lambda>0}|\smallint\limits_{s\sim \lambda^{-1}\tau^{-1}} s^{-2a}\lambda^{-2a}\phi'(s\tau\lambda)^2\lambda^2s^2\tau^2  \frac{ds}{s}|\\
    &\leq \sup_{\lambda>0}\smallint\limits_{s\sim \lambda^{-1}\tau^{-1}} s^{-2a}\lambda^{-2a}\lambda^2s\tau^2  ds\\
    &\asymp \sup_{\lambda>0} \smallint\limits_{s\sim \lambda^{-1}\tau^{-1}} \lambda^{2a}\tau^{2a}\lambda^{-2a}\lambda^2\lambda^{-1}\tau^{-1}\tau^{2}ds\\
    &= \sup_{\lambda>0}\lambda\tau^{2a+1} \smallint\limits_{s\sim \lambda^{-1}\tau^{-1}}ds\\
    &\asymp \sup_{\lambda>0} \lambda\tau^{2a+1}\lambda^{-1}\tau^{-1}\\
    &=\tau^{2a}.
\end{align*}
So, we have proved that  $I_{2,2}^2\lesssim \tau^{2a}\Vert f\Vert_{L^2(G)}^2,$ from where we have the estimate $I_{2,2}\lesssim \tau^{a}\Vert f\Vert_{L^2(G)}.$ Since $\tau\ll 1,$ we have that $\tau^{a}\lesssim \tau^{a-\varepsilon_0}$ and then $I_{2,2}\lesssim I_{2,1}$ as claimed. Here, finally, we have used that $\varepsilon_0\geq 0.$
\end{proof}

\subsection{The $L^1$-theory for $M_1$}
Let us prove the weak (1,1) type of $M_1.$ This part of the proof will be important in identifying the topological condition in \eqref{topological:condition} that the measure $d\sigma$ must satisfy. To do so, we take inspiration from the original proof of the $ L^p$ result by Stein, when making the corresponding analysis in Littlewood-Paley components. So, the arguments about the weak (1,1) type are classical and we closely are going to follow the accurate approach in Grafakos \cite[Pages 479-480]{Grafakos}, adapted to the nilpotent setting.

First, let us prove the pointwise estimate
\begin{equation}\label{pointwise:HLMF}
    M_1f(x)\lesssim \tau^{Q_0-Q}\mathscr{M}f(x),
\end{equation}
where $\mathscr{M}$ denotes the standard Hardy-Littlewood maximal function (see Folland and Stein \cite[Chapter 2]{FollandStein1982}),
\begin{equation}
     \mathscr{M}f(x)=\sup_{r>0}|B(x,r)|^{-1}\smallint_{B(x,r)}|f(y)|dy.
\end{equation} Using  \eqref{pointwise:HLMF} we shall prove that $M_1$ satisfies the weak (1,1) inequality
\begin{equation}\label{L1:estimate:m1}
 \boxed{   |\{x\in G: M_1f(x)>\frac{\lambda}{2}\}|\lesssim \frac{\tau^{Q_0-Q}}{\lambda}\Vert f\Vert_{L^1(G)}}
\end{equation}  For the proof of \eqref{pointwise:HLMF} it suffices to prove the inequality
\begin{equation}
    |\Phi_\tau*d\sigma(x)|\lesssim C_M\tau^{Q_0-Q}(1+|x|)^{-M},
\end{equation}for any $M>Q.$
Note that $\phi\in \mathscr{S}(G)$ is a function in the Schwartz class, and we have the estimate  $|\Phi(z)|\leq C_M(1+|z|)^{-M}.$ Consequently,
\begin{align*}
     |\Phi_\tau*d\sigma(x)| &=\smallint_{G}|\Phi_\tau(xy^{-1})|d\sigma(y)=\tau^{-Q}\smallint_{G}|\Phi(\tau^{-1}\cdot (xy^{-1}))|d\sigma(y)\\
     &\lesssim_M \tau^{-Q}\smallint_{G}(1+|\tau^{-1}(xy^{-1})|)^{-M}d\sigma(y)\\
     &=\tau^{-Q}\smallint_{G}(1+\tau^{-1}|(xy^{-1})|)^{-M}d\sigma(y).
\end{align*}In particular for any $N>M>Q,$ we also have the estimate
\begin{equation}
     |\Phi_\tau*d\sigma(x)| \leq C_N \tau^{-Q}\smallint_{G}(1+\tau^{-1}|(xy^{-1})|)^{-N}d\sigma(y).
\end{equation} Let $K$ be the support of $d\sigma$ and let us consider the regions
\begin{equation}
    S_{-1}(x)=\{y\in K:\tau^{-1}|xy^{-1}|\leq 1 \},
\end{equation}and for any $r\in \mathbb{N}_0,$ let 
\begin{equation}
    S_{r}(x)= \{y\in K:2^{r}<\tau^{-1}|xy^{-1}|< 2^{r+1} \}=: \mathbb{S}_{2^r\tau}(x).
\end{equation}By hypothesis we have that
\begin{equation}
 \sigma( S_{r}(x))\leq (2^r\tau)^{Q_0}.
\end{equation}
Note also that when $y\in S_{r}(x)\subset K\subset B(e,R),$ where $R>0$ is larger  enough,   one has the estimate $|xy^{-1}|\leq 2^r\tau,$ and then from the inequality $|x|-|y|\leq |xy^{-1}|,$ we deduce that $|x|\leq 2^r\tau+R. $ 
Let $j_0\in \mathbb{Z}$ be such that $2^{-j_0}\leq \tau<2^{-j_0+1}.$ From the observations above we have that
\begin{align*}
    &\smallint_{G}(1+\tau^{-1}|(xy^{-1})|)^{-N}d\sigma(y)\\
    &= \sum_{r=-1}^{\infty}\smallint_{S_{r}(x)} (1+\tau^{-1}|(xy^{-1})|)^{-N}d\sigma(y)\\
    &=\sum_{r=-1}^{j_0}\smallint_{S_{r}(x)} (1+\tau^{-1}|(xy^{-1})|)^{-N}d\sigma(y)+\sum_{r=j_0+1}^{\infty}\smallint_{S_{r}(x)} (1+\tau^{-1}|(xy^{-1})|)^{-N}d\sigma(y)\\
     &\sim  \sum_{r=-1}^{j_0}\smallint_{S_{r}(x)} (1+2^{r})^{-N}\chi_{B(e,2^{r}\tau+R)}(x)d\sigma(y)\\
     &\hspace{2cm}+\sum_{r=j_0+1}^{\infty}\smallint_{S_{r}(x)} (1+2^{r})^{-N}\chi_{B(e,2^{r}\tau+R)}(x)d\sigma(y)\\
     & \sim  \sum_{r=-1}^{j_0}\smallint_{S_{r}(x)} (1+2^{r})^{-N}\chi_{B(e,R+2)}(x)d\sigma(y)+\\
     &\hspace{2cm}\sum_{r=j_0+1}^{\infty}\smallint_{S_{r}(x)} (1+2^{r})^{-N}\chi_{B(e,2^{r}\tau+R)}(x)d\sigma(y)\\
     & \lesssim  \sum_{r=-1}^{j_0} 2^{-Nr}\chi_{B(e,R+2)}(x)\sigma(S_r(x))\\
     &\hspace{2cm}+\sum_{r=j_0+1}^{\infty} 2^{-rN}\chi_{B(e,2^{r}\tau+R)}(x)\sigma(S_r(x))\\
     & \leq   \sum_{r=-1}^{j_0} 2^{-Nr}\chi_{B(e,R+2)}(x)\sigma(S_r(x))\\
     &\hspace{2cm}+\sum_{r=j_0+1}^{\infty} 2^{-rN}\chi_{B(e,2^{r}\tau+R)}(x)\sigma(K)\\
     & \lesssim   \sum_{r=-1}^{j_0} 2^{-Nr}\chi_{B(e,R+2)}(x) (2^r\tau)^{Q_0}+\sum_{r=j_0+1}^{\infty} 2^{-rN}\chi_{B(e,2^{r}\tau+R)}(x)\sigma(K).
\end{align*}Note that, allowing $N>M>Q_0,$ we can estimate
\begin{align*}
    \tau^{-Q} \sum_{r=-1}^{j_0} 2^{-Nr}(2^r\tau)^{Q_0}&=\tau^{Q_0-Q} \sum_{r=-1}^{j_0} 2^{-(N-Q_0)r} 
     \lesssim \tau^{Q_0-Q} .
\end{align*}On the other hand, for $x\in B(e,2^{r}\tau+R),$ one has the estimate $|x|\leq 2^{r}\tau+R, $ from where we have the inequality
$$ \chi_{B(e,R+2)}(x) \leq \frac{(1+2^{r}\tau+R)^{M}}{(1+|x|)^{M}} . $$
Consequently,
\begin{align*}
     &\tau^{-Q}\sum_{r=j_0+1}^{\infty} 2^{-rN}\chi_{B(e,2^{r}\tau+R)}(x)\sigma(K)\\
     &\lesssim  \tau^{-Q}\sum_{r=j_0+1}^{\infty} \frac{2^{-rN}(1+2^{r}\tau+R)^{M}}{(1+|x|)^{M}}\\
     &\lesssim_K \frac{1}{(1+|x|)^{M}} \sum_{r=j_0+1}^{\infty} 2^{j_0 Q} 2^{-rN}2^{r M}\tau^{M}\\
     &\sim  \frac{1}{(1+|x|)^{M}} \sum_{r=j_0+1}^{\infty} 2^{j_0 Q} 2^{-rN}2^{r M}2^{-j_0M}\\
     &\sim  \frac{1}{(1+|x|)^{M}} \sum_{r=j_0+1}^{\infty} 2^{-j_0(M-Q)} 2^{-r(N-M)}\\
     &\sim  \frac{2^{j_0(Q-Q_0)}}{(1+|x|)^{M}} \sum_{r=j_0+1}^{\infty} 2^{-j_0(M-Q+Q-Q_0)} 2^{-r(N-M)}.
\end{align*}Since $M-Q_0>0,$ and $N>M,$ we have that  $$\sum_{r=j_0+1}^{\infty} 2^{-j_0(M-Q_0)} 2^{-r(N-M)}\lesssim 1.$$
Putting all these estimates together, we have proved that
\begin{align*}    
     |\Phi_\tau*d\sigma(x)| &\lesssim_N \tau^{-Q}\smallint_{G}(1+\tau^{-1}|(xy^{-1})|)^{-N}d\sigma(y)\\
     & \lesssim \tau^{Q_0-Q}\chi_{B(e,R+2)}+  \frac{2^{j_0(Q-Q_0)}}{(1+|x|)^{M}}\\
     &\lesssim_{M}  \frac{\tau^{Q_0-Q}}{(1+|x|)^{M}}+  \frac{\tau^{Q_0-Q}}{(1+|x|)^{M}}=\frac{2\tau^{Q_0-Q}}{(1+|x|)^{M}}=\Psi(x).
\end{align*}Since $\Phi_\tau*d\sigma(x)$ has been dominated by the radially `decreasing function'\footnote{in the sense that when $|x|\leq |y|,$ $\Psi(x)\geq \Psi(y).$} $\Psi(x),$
 we have the inequality
\begin{equation}\label{Weak:1:1:inequality}
2\tau^{Q_0-Q} \sup_{t>0}|f\ast \Psi_t(x)|\lesssim   2\tau^{Q_0-Q} \mathscr{M}f(x),
\end{equation} as claimed. Indeed, \eqref{Weak:1:1:inequality} follows from the weak (1,1) type of the Hardy-Littlewood maximal function $\mathscr{M}$ on a homogeneous (nilpotent) Lie group, see Folland and Stein \cite[Chapter 2]{FollandStein1982}.

\subsection{The restricted weak estimate for $M_{F}^{d\sigma}$}\label{proof:subsection:full}
Let $A\subset G$ be a Borel measurable set and let $f=\chi_A$ be its characteristic function.  We are going to prove the estimate
\begin{align}
     \boxed{|\{x\in G: M_F^{d\sigma}f(x)>\lambda\}|\lesssim {\lambda^{-\frac{D}{D-c}}}|A|}
\end{align}where $D:=2a-\varepsilon_0+2c.$
\begin{proof}[Proof of Theorem \ref{Maximal:Function:Graded}] In view of \eqref{M:2:op:norm} and of \eqref{L1:estimate:m1}, 
 the weak (1,1) boundedness of $M_1$ together with the $L^2$-boundedness of $M_2$ provides the following estimates when $f=\chi_{A},$
\begin{equation}
    |\{x\in G: M_1f(x)>\frac{\lambda}{2}\}|\lesssim \frac{\tau^{Q_0-Q}}{\lambda}|A|
\end{equation} and 
\begin{equation}
\left|\left\{x\in G: M_2f(x)>\frac{\lambda}{2}\right\}\right|
\leq\smallint\limits_{\{x\in G: M_2f(x)>\frac{\lambda}{2}\}} \frac{M_2f(x)^2}{\lambda^2}dx
\lesssim \frac{\tau^{2a-\varepsilon_0}}{\lambda^2}|A|.
\end{equation} So, we have proved that
\begin{align*}
     |\{x\in G: M_F^{d\sigma}f(x)>\lambda\}|\lesssim \left(\frac{\tau^{Q_0-Q}}{\lambda}+\frac{\tau^{2a-\varepsilon_0}}{\lambda^2}\right)|A|.
\end{align*} Now, let us make an optimization procedure on the constant on the right-hand side.  Let us consider the function
\begin{align*}
    \varkappa(\tau):=\frac{\tau^{Q_0-Q}}{\lambda}+\frac{\tau^{2a-\varepsilon_0}}{\lambda^2}.
\end{align*}Let  $\lambda< 1.$ Since $Q-Q_0>0,$ we have that   
\begin{equation}
    \tau_0=\left(\frac{Q-Q_0}{2a-\varepsilon_0}\lambda\right)^{\frac{1}{2a-\varepsilon_0+Q-Q_0}},
\end{equation} is a critical point of $\varkappa(\tau),$ namely, it satisfies that $\varkappa'(\tau_0)=0.$ We also observe that   
\begin{equation}
     \varkappa(\tau_0)=C(a,\varepsilon_0,Q,Q_0)\lambda^{ -\frac{2a-\varepsilon_0+2c}{2a-\varepsilon_0+c} },\, c=Q-Q_0,
\end{equation} where the constant $C(a,\varepsilon_0,Q,Q_0)$ depends on $a,\varepsilon_0,Q,Q_0,$ but not on $\lambda.$
A direct computation shows that 
\begin{equation*}
    C(a,\varepsilon_0,Q,Q_0)=F^{\frac{2a-\varepsilon_0}{2a-\varepsilon_0+Q-Q_0}}+F^{\frac{Q_0-Q}{2a-\varepsilon_0+Q-Q_0}},\,F:=\frac{Q-Q_0}{2a-\varepsilon_0}.
\end{equation*}Moreover, 
\begin{equation}
    \varkappa''(\tau_0)=F^{\frac{2a-\varepsilon_0-2}{2a-\varepsilon_0+Q-Q_0}}(2a-\varepsilon_0)(2a-\varepsilon_0+Q-Q_0)>0,
\end{equation}showing that $\tau=\tau_0$ is a local minimum of $\varkappa(\tau).$ Because of the assumption $0<\lambda < 1,$ and since  $2a-\varepsilon_0+Q-Q_0>0,$ we can choose $\tau=\tau_0$ in our analysis above, that automatically will satisfy the condition $\tau^{-1}\gg 1$ imposed at the beginning. So, this choice of $\tau$ implies the estimate
\begin{align}\label{Weak:inequality:proof}
     |\{x\in G: M_F^{d\sigma}f(x)>\lambda\}|\lesssim {\lambda^{-\frac{D}{D-c}}}|A|,
\end{align}where $D=2a-\varepsilon_0+2c.$
We observe that only the inequality in \eqref{Weak:inequality:proof} when $\lambda< 1$ is relevant for the proof. Indeed,  by the Riesz-representation theorem we can write $d\sigma=Kdx,$ where $dx$ denotes the Haar measure on $G.$ In consequence, we have that
\begin{align*}
   & M_{F}^{d\sigma}(\chi_A)(x)\\
    &=\sup_{t>0}|\chi_A\ast K_t(x)|\leq \sup_{t>0}\Vert \chi_A\ast K_t \Vert_{L^\infty}\leq \Vert\chi_A\Vert_{L^\infty}\sup_{t>0}\Vert K_t\Vert_{L^1}=\Vert K\Vert_{L^1(G)}=\sigma(G).
\end{align*}Without loss of generality if we consider the measure $d\sigma$ to be normalised, then we have the inequality $M_{F}^{d\sigma}(\chi_A)(x)\leq 1.$ Consequently if $\lambda>1,$ $|\{x\in G: M_F^{d\sigma}f(x)>\lambda\}|=0.$
The estimate \eqref{Weak:inequality:proof}  proves that $M_F^{d\sigma}$ is of restricted weak $(\frac{D}{D-c},1)$ type. The proof of Theorem \ref{Maximal:Function:Graded} is complete.
\end{proof}

\section{Final remarks}\label{FinalRemarks}
\subsubsection{About our main result} This paper aims to extend the approach of Bourgain to analyse the boundedness of the full maximal function in a class of homogeneous nilpotent Lie groups, admitting the existence of left-invariant hypoelliptic partial differential operators, and then we assume our group to be a graded Lie group in view of the Helffer and Nourrigat solution in \cite{HelfferNourrigat} of the Rockland conjecture. Our main result is the restricted weak $(p,1)$ estimate presented in Theorem \ref{Maximal:Function:Graded}.

\subsubsection{Applications: differentiability of functions on graded Lie groups}  We also note, as expected, that the boundedness properties of maximal functions, and in particular our main Theorem \ref{Maximal:Function:Graded}, have an immediate consequence in the differentiations of functions in the Lebesgue spaces $L^p(G)$, extending the classical Stein differentiation theorem, see Remark \ref{differentiation:th}. An excellent introduction to this topic is the one given by  Stein in \cite{Stein1970}.

\subsubsection{Relation of our approach with the global quantisation on graded Lie groups}
As expected, there arise several difficulties in proving our main Theorem \ref{Maximal:Function:Graded}. One of the main obstacles is to deal with the non-commutative setting of nilpotent Lie groups $G$. In this setting, the Fourier transform is operator-valued taking values at the unitary and irreducible representations of the group, that are elements of the unitary dual $\widehat{G}$. 
Nevertheless, here we follow one of the recent developments of the theory of pseudo-differential operators on nilpotent Lie groups, for instance, the {\it global quantisation on graded Lie groups}. 

According to this research perspective the symbol classes and the boundedness properties of operators can be formulated in terms of criteria involving orders measured by the symbols of Rockland operators. Indeed, this approach has been developed recently in the theory of pseudo-differential operators of Fischer and Ruzhansky \cite{FischerRuzhanskyBook}, as a global extension to the local approach developed by H\"ormander \cite{HormanderBook34}. For the aspects of the global theory of pseudo-differential operators on compact Lie groups we refer to Ruzhansky, Turunen and Wirth \cite{Ruz, RuzhanskyWirth2014} and to \cite{CardonaThesis,CerejeirasFerreiraKahlerWirth}. 

So, motivated by this new research perspective that by itself is a source of open problems, particularly in harmonic analysis and PDE, in our analysis of the full maximal function, we have formulated our hypotheses for a class of measures whose Fourier transform is closely related to the H\"ormander symbol classes developed in \cite{FischerRuzhanskyBook}, see Definition \ref{Admissible:measure}. Then, a criterion on the corresponding measure $d\sigma$ has been formulated in Theorem  \ref{Maximal:Function:Graded} in order to guarantee the boundedness of $M_F^{d\sigma}$

\subsubsection*{Conflict of interests statement - Data Availability Statements}  The  author states that there is no conflict of interest.  Data sharing does not apply to this article as no datasets were generated or
analysed during the current study.

\subsection*{Acknowledgements} The author is deeply indebted to Michael Ruzhansky and Julio Delgado for the support, encouragement, and guidance given over the years; also for several discussions related to the harmonic analysis on graded Lie groups.  The author is currently a FWO fellow supported by the Research Foundation Flanders FWO.

\bibliographystyle{amsplain}

\end{document}